\documentclass[11pt,a4paper]{amsart}
\usepackage{amsmath}
\usepackage[utf8]{inputenc}
\usepackage[english]{babel}
\usepackage[T1]{fontenc}
\usepackage{url}
\usepackage{mathrsfs}
\usepackage{ifthen}
\usepackage{tikz}
\usepackage{ amssymb }
\usepackage{lscape}
\usepackage{multirow}

\newcommand \ifdraw {\iffalse}

\newcommand \Mset  {\mathfrak{M}}     

\newcommand \M  {\mathcal{M}}       
\newcommand \Mirr  {\mathcal{M}^{(Ir)}}  
\newcommand \MIS  {\mathcal{M}^{(IS)}} 
\newcommand \MS  {\mathcal{M}^{(S)}} 
\newcommand \A  {\mathcal{A}} 

\newcommand \gfM  {\phi}
\newcommand \gfMS  {\phi_{(S)}}
\newcommand \gfMIS  {\phi_{(IS)}}
\newcommand \gfMIrr  {\phi_{(Ir)}}

\newcommand \Submnd{\operatorname{Sub}}

\newtheorem{thm}{Theorem}
\newtheorem{prop}{Proposition}
\newtheorem{lemma}{Lemma}

\newtheorem{corollary}{Corollary}

\theoremstyle{definition}
\newtheorem{definition}{Definition}

\theoremstyle{remark}
\newtheorem{remark}{Remark}

\newtheorem{example}{Example}

\title[Prime Factorization of Meanders]{Prime Factorization of Meanders}
\author{Yury Belousov}
    \thanks{The work is supported by the Ministry of Science and Higher Education of the Russian Federation (agreement no. 075-15-2025-343).}

\address{Yury Belousov\\
Saint Petersburg State University}
\email{bus99@yandex.ru}

\begin{document}
\begin{abstract}
This paper introduces a prime factorization of open meanders, articulated through the framework of 2-colored operads. We demonstrate that each open meander can be canonically constructed from building blocks of two types: iterated snakes and irreducible meanders. We find out that iterated snakes allow efficient enumeration, and thus the problem of enumerating meanders reduces to the problem of enumerating irreducible meanders. Additionally, we present some results concerning the asymptotics of meanders of both classes.

\smallskip
\noindent \textbf{Keywords:} meander,
    operad, 
    plane curve,
    low-dimensional topology.
\end{abstract}

\maketitle

\section{Introduction}
A meander is a configuration of a pair of simple curves in a disk (the formal definitions are given in Section~\ref{sec: basic definitions}). 
Examples of meanders can be viewed in Figure~\ref{fig: examples of meanders} and throughout this paper.

\begin{figure}[h]
    \centering
    \begin{tikzpicture}[scale = 1.3]
        \draw[thick] (0, 0) to (3, 0);
        \draw[ultra thick] (0.15629, 0.666667) to [out = 0, in = 90, distance = 12.5664] (1, 0)
         to [out = -90, in = -90, distance = 12.5664] (2, 0)
        to [out = 90, in = 180, distance = 12.5664] (2.84371, 0.666667, 0);
        \draw[help lines] (1.5, 0) circle (1.5);
        \draw[fill] (0.15629, 0.666667) circle (0.05);
        \draw[fill] (2.84371, 0.666667) circle (0.05);
        \draw[fill] (0, 0) circle (0.05);
        \draw[fill] (3, 0) circle (0.05);
    \end{tikzpicture}
    \hspace{1.5cm}
    \begin{tikzpicture}[scale = 1.3]
        \draw[thick] (0, 0) to (3, 0);
        \draw[ultra thick] (0.163915, 0.681818) to[out = 0, in = 90, distance = 5.14079] (0.409091, 0)
         to[out = -90, in = -90, distance = 1.7136] (0.545455, 0)
         to[out = 90, in = 90, distance = 5.14079] (0.954545, 0)
         to[out = -90, in = -90, distance = 11.9952] (1.90909, 0)
         to[out = 90, in = 90, distance = 1.7136] (2.04545, 0)
         to[out = -90, in = -90, distance = 15.4224] (0.818182, 0)
         to[out = 90, in = 90, distance = 1.7136] (0.681818, 0)
         to[out = -90, in = -90, distance = 18.8496] (2.18182, 0)
         to[out = 90, in = 90, distance = 5.14079] (1.77273, 0)
         to[out = -90, in = -90, distance = 1.7136] (1.63636, 0)
         to[out = 90, in = 90, distance = 8.56798] (2.31818, 0)
         to[out = -90, in = -90, distance = 25.7039] (0.272727, 0)
         to[out = 90, in = 90, distance = 1.7136] (0.136364, 0)
         to[out = -90, in = -90, distance = 29.1311] (2.45455, 0)
         to[out = 90, in = 90, distance = 11.9952] (1.5, 0)
         to[out = -90, in = -90, distance = 5.14079] (1.09091, 0)
         to[out = 90, in = 90, distance = 22.2767] (2.86364, 0)
         to[out = -90, in = -90, distance = 1.7136] (2.72727, 0)
         to[out = 90, in = 90, distance = 18.8496] (1.22727, 0)
         to[out = -90, in = -90, distance = 1.7136] (1.36364, 0)
         to[out = 90, in = 90, distance = 15.4224] (2.59091, 0)
        to[out = -90, in = 180, distance = 5.14079] (2.72643, -0.863636, 0);
        \draw[help lines] (1.5, 0) circle (1.5);
        \draw[fill] (0.163915, 0.681818) circle (0.05);
        \draw[fill] (2.72643, -0.863636) circle (0.05);
        \draw[fill] (0, 0) circle (0.05);
        \draw[fill] (3, 0) circle (0.05);
    \end{tikzpicture}
    \caption{Examples of meanders.}
    \label{fig: examples of meanders}
\end{figure}


V.~Arnol'd was the first to use the term ``meander'' and to formulate the problem of counting them~\cite{A88}. A similar problem of counting closed meanders (under the name of planar permutation) was formulated by P.~Rosenstiehl in~\cite{R84}.
A detailed historical background to the subject of meanders can be found in~\cite{Z23}; however, this section briefly highlights some connections between meanders and other areas of mathematics and physics. Notable associations include the Temperley--Lieb algebras (see~\cite{DFGG97}), invariants of 3-manifolds (see~\cite{KS91}), models of statistical physics (see~\cite{DFGG00}), parabolic PDEs (see~\cite{FR91}), and moduli spaces of meromorphic quadratic differentials (see~\cite{DGZZ20}). 
The theory of meanders has been actively developed: various approaches to the calculation of meanders have been proposed (see, for example,~\cite{J00}, \cite{FE02}, \cite{BS10}, \cite{L22}) and the study of the asymptotic behavior of meander numbers (\cite{DFGG00}, \cite{AP05}). In the recent article \cite{DGZZ20} the exact asymptotics of the number of closed meanders with a fixed number of minimal arcs were obtained.

Despite the high interest in this area, the key questions remain open. The number of meanders with a given number of intersections is unknown, as is the asymptotic behavior of these numbers.

In this paper, we develop a new approach to the theory of meanders. 
We show that each meander can be canonically decomposed into prime factors of two different types: iterated snakes and irreducible meanders (Theorem~\ref{thm: main}). We also study each of these classes. It turns out that iterated snakes are fairly simple objects: we can explicitly write out their generating function (Theorem~\ref{thm: generating function of iterated snakes}), and we can enumerate them very efficiently (Corollary~\ref{cor: formula for iterated snakes}). In addition, we have discovered the connection of iterated snakes to another open combinatorial problem --- the enumeration of polyominoes (see Section~\ref{sec: iterated snakes connections}).
In contrast, irreducible meanders are mysterious and we have been able to find out only some of their basic properties. 

The decomposition of meanders mirrors a phenomenon that occasionally occurs in the classification of low-dimensional topological objects, where typically, two primary classes of fundamental objects are identified: one elementary and the other more complex. All other objects are then constructed from these prime elements through specific operations, which can be described within the operadic framework. For instance, in knot theory, each knot can be classified as either a torus (representing the simpler prime objects) or hyperbolic (representing the complex prime objects) or can be constructed from them using a satellite operation (expressible in terms of operads, see~\cite{B12}). Similarly, in the theory of braid groups: up to conjugation each braid is either periodic, or pseudo-Anosov, or can be constructed from them using a cabling operation. This classification follows the famous Nielsen--Thurston classification (see, for example,~\cite{T22}); for details on braid operads see~\cite{Y21}.


The paper is structured as follows. In Section~\ref{sec: basic definitions} we give basic definitions related to meanders; in  Section~\ref{sec: decomposition} we discuss the factorization; in Section~\ref{sec: iterated snakes} we focus on iterated snakes, and, finally, in Section~\ref{sec: irreducible meanders} we study irreducible meanders. Some computational results are given in Appendix~\ref{app: table}. More numerical data, as well as the code used to derive them, can be found in~\cite{Bcode}.

\subsection*{Acknowledgment}
The author thanks Andrei Malyutin, Ivan Dynnikov, and Ilya Alekseev for helpful suggestions and comments. The author is also grateful to the anonymous referee for a careful reading and valuable suggestions.

\section{Basic definitions} \label{sec: basic definitions}
\begin{definition}\label{def: meander}
A \emph{singular meander} $(D, (p_1, p_2,p_3,p_4), (m, l))$ is a triple of \begin{itemize}
    \item Euclidean 2-dimensional disk $D$;
    \item four distinct points $p_1, p_2,p_3,p_4$ on the boundary $\partial D$ such that there exists a connected component of $\partial D \setminus \{p_1, p_2\}$ containing $\{p_3,p_4\}$;
    \item the images $m$ and $l$ of smooth proper embeddings of the segment $[0;\,1]$ into $D$ such that $\partial m = \{p_1,p_3\}$, $\partial l = \{p_2,p_4\}$, and $m$ and $l$ intersect (not necessary transversely) in a non-zero finite number of points.
\end{itemize}
The intersection points of $m$ and $l$ are called \emph{intersections of $M$}. 
\end{definition}

\begin{remark}
    Usually, only meanders with transverse intersections are considered. However, it is convenient for us to extend the class of objects under consideration. In the rest of the paper we omit the word ``singular''. If we wanted to emphasize the fact that a given meander has only transverse intersections, we would say ``non-singular meander''.
\end{remark}

\begin{remark}
    What we call a ``meander'' is usually referred to in modern literature as an ``open meander'', whereas a ``meander'' is a pair of closed curves in a disk. Our definition is the same as Arnold's original. Furthermore, it is not possible to apply our technique to closed meanders without considering open meanders.
\end{remark}

\begin{definition}
We say that two meanders $$M=(D,(p_1,p_2,p_3,p_4), (m,l))$$ and $$M'= (D',(p_1',p_2',p_3',p_4'),(m',l'))$$ are \emph{equivalent} if there exists a homeomorphism $f:D\to D'$ such that $f(m)=m'$, $f(l)=l'$, and $f(p_i)=p_i'$ for each $i=1,\dots,4$. 
\end{definition}

\begin{remark}
    We now explain why we include the boundary points $p_1,p_2,p_3,p_4$ and impose the additional restriction --- that there exists a connected component of $\partial D\setminus\{p_1,p_2\}$ containing $\{p_3,p_4\}$ --- in the definition of meanders. If we do not regard the boundary points as part of the tuple, the meanders in Figure~\ref{fig: boundary points remark} would be equivalent (since they are related by a reflection across the vertical diameter of the disk and by isotopy). If we instead include the boundary points in the definition but do not impose the additional restriction, then the left meander in Figure~\ref{fig: examples of meanders} admits two non-equivalent interpretations, depending on whether the upper-left boundary point is labeled $p_1$ or $p_3$ (on both figures $l$ is the horizontal diameter).
    \begin{figure}[h]
        \centering
       \begin{tikzpicture}[scale = 3.9]
\draw[thick] (0, 0) to (1, 0);
\draw[ultra thick] (0.0285955, 0.166667)
	to[out = 0, in = 90, distance = 10.472] (0.833333, 0)
	to[out = -90, in = -90, distance = 2.09439] (0.666667, 0)
	to[out = 90, in = 90, distance = 2.09439] (0.5, 0)
	to[out = -90, in = -90, distance = 2.09439] (0.333333, 0)
	to[out = 90, in = 90, distance = 2.09439] (0.166667, 0)
	to[out = -90, in = 180, distance = 10.472] (0.971405, -0.166667);
\draw[help lines] (0.5, 0) circle (0.5);
\draw[fill] (0.0285955, 0.166667) circle (0.0166667);
\draw[fill] (0.971405, -0.166667) circle (0.0166667);
\draw[fill] (0, 0) circle (0.0166667);
\draw[fill] (1, 0) circle (0.0166667);
\end{tikzpicture}
    \hspace{2cm}
\begin{tikzpicture}[scale = 3.9]
\draw[thick] (0, 0) to (1, 0);
\draw[ultra thick] (0.0285955, 0.166667)
	to[out = 0, in = 90, distance = 2.09439] (0.166667, 0)
	to[out = -90, in = -90, distance = 2.09439] (0.333333, 0)
	to[out = 90, in = 90, distance = 2.09439] (0.5, 0)
	to[out = -90, in = -90, distance = 2.09439] (0.666667, 0)
	to[out = 90, in = 90, distance = 2.09439] (0.833333, 0)
	to[out = -90, in = 180, distance = 2.09439] (0.971405, -0.166667);
\draw[help lines] (0.5, 0) circle (0.5);
\draw[fill] (0.0285955, 0.166667) circle (0.0166667);
\draw[fill] (0.971405, -0.166667) circle (0.0166667);
\draw[fill] (0, 0) circle (0.0166667);
\draw[fill] (1, 0) circle (0.0166667);
\end{tikzpicture}
        \caption{Meanders that would be equivalent if boundary points are not part of the data.}
        \label{fig: boundary points remark}
    \end{figure}
\end{remark}

\begin{remark}
    Throughout the figures we fix a drawing convention: we identify $D$ with a Euclidean disk, draw $l$ as a horizontal diameter with $p_2$ at the left end, and, after possibly reflecting across $l$, place $p_1$ above $p_2$. This uses that meanders are taken up to homeomorphism fixing the labeled boundary points and does not restrict generality. That is why we do not place $l$, $m$, $p_1,p_2,p_3,p_4$ in the figures. Examples of meanders with non-transverse intersections are given in Figure~\ref{fig: examples of singular meanders}.
\end{remark}

\begin{figure}[h]
    \centering
    \begin{tikzpicture}[scale = 3.9]
\draw[thick] (0, 0) to (1, 0);
\draw[ultra thick] (0.0919029, 0.288889)
	to[out = 0, in = 90, distance = 1.67552] (0.133333, 0)
	to[out = -90, in = -90, distance = 5.86431] (0.6, 0)
	to[out = -90, in = -90, distance = 6.70206] (0.0666667, 0)
	to[out = -90, in = -90, distance = 7.53982] (0.666667, 0)
	to[out = 90, in = 90, distance = 1.67552] (0.533333, 0)
	to[out = -90, in = -90, distance = 4.18879] (0.2, 0)
	to[out = 90, in = 90, distance = 9.21534] (0.933333, 0)
	to[out = 90, in = 90, distance = 8.37758] (0.266667, 0)
	to[out = -90, in = -90, distance = 2.51327] (0.466667, 0)
	to[out = -90, in = -90, distance = 1.67552] (0.333333, 0)
	to[out = 90, in = 90, distance = 5.86431] (0.8, 0)
	to[out = 90, in = 90, distance = 5.02655] (0.4, 0)
	to[out = 90, in = 90, distance = 4.18879] (0.733333, 0)
	to[out = -90, in = -90, distance = 1.67552] (0.866667, 0)
	to[out = -90, in = 180, distance = 1.67552] (0.936173, -0.244444);
\draw[help lines] (0.5, 0) circle (0.5);
\draw[fill] (0.0919029, 0.288889) circle (0.0166667);
\draw[fill] (0.936173, -0.244444) circle (0.0166667);
\draw[fill] (0, 0) circle (0.0166667);
\draw[fill] (1, 0) circle (0.0166667);
\end{tikzpicture}
\hspace{2cm}
\begin{tikzpicture}[scale = 3.9]
\draw[thick] (0, 0) to (1, 0);
\draw[ultra thick] (0.0208426, 0.142857)
	to[out = 0, in = 90, distance = 10.7712] (0.857143, 0)
	to[out = -90, in = -90, distance = 1.7952] (0.714286, 0)
	to[out = -90, in = -90, distance = 1.7952] (0.571429, 0)
	to[out = 90, in = 90, distance = 1.7952] (0.428571, 0)
	to[out = -90, in = -90, distance = 1.7952] (0.285714, 0)
	to[out = -90, in = -90, distance = 1.7952] (0.142857, 0)
	to[out = -90, in = 180, distance = 10.7712] (0.979157, -0.142857);
\draw[help lines] (0.5, 0) circle (0.5);
\draw[fill] (0.0208426, 0.142857) circle (0.0166667);
\draw[fill] (0.979157, -0.142857) circle (0.0166667);
\draw[fill] (0, 0) circle (0.0166667);
\draw[fill] (1, 0) circle (0.0166667);
\end{tikzpicture}

    \caption{Examples of singular meanders.}
    \label{fig: examples of singular meanders}
   
\end{figure}

\begin{definition}
Let $M=(D, (p_1, p_2,p_3,p_4), (m, l))$ be a meander, let $n_{\mathrm{t}}(M)$ be the number of transverse intersections of $m$ and $l$, and let $n_{\mathrm{nt}}(M)$ be the number of non-transverse intersections of $m$ and $l$. The \emph{order of $M$} is the pair $(n, k) := \left(\max_{M'\in [M]} n_{\mathrm{t}}(M'), \min_{M'\in [M]} n_{\mathrm{nt}}(M')\right)$, where $[M]$ is the set of all meanders equivalent to $M$. 
If the order of $M$ is $(n, k)$ then the \emph{total order} of $M$ is $n+k$. By $\Mset_{n, k}$ we denote the set of all equivalence classes of meanders of order $(n, k)$, and by $\M_{n,k}$ we denote the cardinality of $\Mset_{n,k}$.
\end{definition}

\begin{remark}
    We present this definition to avoid non-transverse intersections that can be made transverse by a small isotopy --- as, for example, the intersection of the graph $y=x^3$ with the $x$-axis $\{y=0\}$.
\end{remark}

Without loss of generality we always assume that if $M$ is a meander of order $(n,k)$, then $n_{\mathrm{t}}(M)=n$ and $n_{\mathrm{nt}}(M) = k$.

\subsection{Meander permutation}
It is convenient to work with meanders using their representation via permutations. To do this, we attach a permutation to each meander as follows. 
Let  
$$M = (D, (p_1, p_2,p_3,p_4), (m, l))$$
be a meander of order $(n, k)$. Consider a bijective map ${\gamma: [0;\,n+k+1] \to l}$, such that 
\begin{enumerate}
    \item  $\gamma(0) = p_2$,
    \item $\gamma(t)$ is an intersection of $M$ if and only if $t \in \{1,2,\dots n+k\}$. We say that an intersection $p$ has the \emph{label} $a$ if $\gamma(a) = p$.
\end{enumerate}
If we write the labels in the order of movement from $p_1$ to $p_3$ along $m$, we get the \emph{permutation of $M$}. Note that the labels do not depend on the choice of $\gamma$, so the permutation of $M$ is well-defined. 
For example, the permutation of the second meander in Figure~\ref{fig: examples of singular meanders} is $(6,5,4,3,2,1)$. 

\begin{remark}\label{rmrk: permutations determines meanders}
    For non-singular meanders, the associated permutation uniquely determines the meander; by contrast, for meanders with non-transverse intersections a permutation alone does not suffice (in this case we need additional information on which intersections are transverse, and which are not). 
\end{remark}

\subsection{Submeanders, and the insertion of meanders}
The key ingredient in the meander factorization is the notion of submeanders (see Definition~\ref{def: submeander}). For a given meander $M$ and its submeander $M'$, there is a canonical procedure for cutting $M'$ from $M$. So, informally speaking, by choosing certain submeanders and cutting them from $M$ we get a canonical procedure for decomposing a meander into simpler pieces.

\begin{definition} \label{def: submeander}
We say that a meander $M'=(D', (p_1', p_2',p_3',p_4'), (m', l'))$ is a \emph{submeander} of a meander $M=(D, (p_1, p_2,p_3,p_4), (m, l))$ if 
\begin{itemize}
    \item $D' \subseteq D$.
    \item $m' = D' \cap m$.
    \item $l' = D' \cap l$.
    \item $p_1' = \gamma_m\left(t_1\right)$, where  $\gamma_m: [0;\,1] \to D$ is any injective continuous map such that $\gamma_m([0;\,1]) = m$, $\gamma_m(0) = p_1$, and $t_1=\min \{t\in [0;\,1]\ |\ \gamma_m(t) \in D'\}$.
    \item Let $S=\gamma_m^{-1}(l)\cap[0;\,t_1].$
    If $S\neq\varnothing$, let $t_q=\max S$, and $q=\gamma_m(t_q)$; otherwise set $q:=p_2$. Choose an injective continuous map $\gamma_l:[0;\,1]\to D$ with $\gamma_l([0;\,1])=l$ such that $\gamma_l^{-1}(q)<t$ for all $t\in \gamma_l^{-1}(D')$. Then $p_2'$ is defined as $p_2'=\gamma_l(t_2)$, where $t_2=\min\{\,t\in[0;\,1]\ |\ \gamma_l(t)\in D'\,\}$.
\end{itemize}
\end{definition}

\begin{definition}
Let $$M'=(D', (p_1', p_2',p_3',p_4'), (m', l'))$$ and  $$M''=(D'', (p_1'', p_2'',p_3'',p_4''), (m'', l''))$$ be two submeanders of a meander $$M=(D, (p_1, p_2,p_3,p_4), (m, l)).$$ We say that $M'$ and $M''$ are \emph{equivalent with respect to $M$} if $D'\cap m\cap l = D'' \cap m \cap l$.
\end{definition}
\begin{remark}
Note that two submeanders $M'$ and $M''$ of $M$ can be equivalent as meanders but not equivalent as submeanders with respect to $M$. 
See Figure~\ref{fig: examples of submeanders}: all highlighted submeanders are mutually equivalent as meanders, yet they are not equivalent as submeanders of $M$.
\end{remark}

\begin{figure}[h]
    \centering
    \begin{tikzpicture}[scale = 4]
\draw[thick] (0, 0) to (1, 0);
\draw[ultra thick] (0.0740823, 0.261905)
	to[out = 0, in = 90, distance = 2.69279] (0.214286, 0)
	to[out = -90, in = -90, distance = 0.897598] (0.142857, 0)
	to[out = 90, in = 90, distance = 0.897598] (0.0714286, 0)
	to[out = -90, in = -90, distance = 4.48799] (0.428571, 0)
	to[out = 90, in = 90, distance = 0.897598] (0.357143, 0)
	to[out = -90, in = -90, distance = 0.897598] (0.285714, 0)
	to[out = 90, in = 90, distance = 8.07838] (0.928571, 0)
	to[out = -90, in = -90, distance = 2.69279] (0.714286, 0)
	to[out = 90, in = 90, distance = 0.897598] (0.785714, 0)
	to[out = -90, in = -90, distance = 0.897598] (0.857143, 0)
	to[out = 90, in = 90, distance = 4.48799] (0.5, 0)
	to[out = -90, in = -90, distance = 0.897598] (0.571429, 0)
	to[out = 90, in = 90, distance = 0.897598] (0.642857, 0)
	to[out = -90, in = 180, distance = 4.48799] (0.971405, -0.166667);
\draw[help lines] (0.5, 0) circle (0.5);
\draw[fill] (0.0740823, 0.261905) circle (0.0166667);
\draw[fill] (0.971405, -0.166667) circle (0.0166667);
\draw[fill] (0, 0) circle (0.0166667);
\draw[fill] (1, 0) circle (0.0166667);

\draw[fill, gray, opacity = 0.5] (2/14, 0) circle (1.2/13);
\draw[fill, gray, opacity = 0.5] (5/14, 0) circle (1.2/13);
\draw[fill, gray, opacity = 0.5] (8/14, 0) circle (1.2/13);
\draw[fill, gray, opacity = 0.5] (11/14, 0) circle (1.2/13);
\end{tikzpicture}
    \caption{Examples of submeanders.}
    \label{fig: examples of submeanders}
\end{figure}

\begin{definition}
Let $M$ be a meander, and let $M_1$ and $M_2$ be two of its submeanders. We say that $M_1 \leq M_2$ if there exists $M'_1$ --- a submeander of $M$ that is equivalent to $M_1$ with respect to $M$, such that $M'_1$ is also a submeander of $M_2$. \\
Thus there is a well-defined partial order on the set of all submeanders of $M$ up to equivalence with respect to $M$; we denote this set by $\Submnd(M)$. 
\end{definition}

Figure~\ref{fig: example of graphs} shows examples of meanders and the Hasse diagram of their poset of submeanders. 
\begin{figure}[h]
\begin{tikzpicture}
\draw[thick] (0, 0) to (3, 0);
\draw[ultra thick] (0.0857864, 0.5) to[out = 0, in = 90, distance = 26.3894] (2.1, 0)
 to[out = -90, in = -90, distance = 3.76991] (2.4, 0)
 to[out = 90, in = 90, distance = 3.76991] (2.7, 0)
 to[out = -90, in = -90, distance = 11.3097] (1.8, 0)
 to[out = 90, in = 90, distance = 3.76991] (1.5, 0)
 to[out = -90, in = -90, distance = 11.3097] (0.6, 0)
 to[out = 90, in = 90, distance = 3.76991] (0.9, 0)
 to[out = -90, in = -90, distance = 3.76991] (1.2, 0)
 to[out = 90, in = 90, distance = 11.3097] (0.3, 0)
to[out = -90, in = 180, distance = 33.9292] (2.91421, -0.5, 0);
\draw[help lines] (1.5, 0) circle (1.5);
\draw[fill] (0.0857864, 0.5) circle (0.05);
\draw[fill] (2.91421, -0.5) circle (0.05);
\draw[fill] (0, 0) circle (0.05);
\draw[fill] (3, 0) circle (0.05);
\end{tikzpicture}
\begin{tikzpicture}[scale = 1]
\draw[fill] (1.33333, 2.66667) circle (0.045);
\draw[fill] (1.33333, 2.66667) to (1.5, 2.33333);
\draw[fill] (1.33333, 2.66667) to (0.833333, 1.66667);
\draw[fill] (1.5, 2.33333) circle (0.045);
\draw[fill] (1.5, 2.33333) to (2, 1.33333);
\draw[fill] (1.5, 2.33333) to (1, 1.33333);
\draw[fill] (0.833333, 1.66667) circle (0.045);
\draw[fill] (0.833333, 1.66667) to (1, 1.33333);
\draw[fill] (0.833333, 1.66667) to (0.666667, 1.33333);
\draw[fill] (2, 1.33333) circle (0.045);
\draw[fill] (2, 1.33333) to (2.16667, 1);
\draw[fill] (2, 1.33333) to (1.5, 0.333333);
\draw[fill] (1, 1.33333) circle (0.045);
\draw[fill] (1, 1.33333) to (0.833333, 1);
\draw[fill] (1, 1.33333) to (1.5, 0.333333);
\draw[fill] (0.666667, 1.33333) circle (0.045);
\draw[fill] (0.666667, 1.33333) to (0.833333, 1);
\draw[fill] (0.666667, 1.33333) to (0.5, 1);
\draw[fill] (2.16667, 1) circle (0.045);
\draw[fill] (2.16667, 1) to (2.33333, 0.666667);
\draw[fill] (2.16667, 1) to (1.66667, 0);
\draw[fill] (0.833333, 1) circle (0.045);
\draw[fill] (0.833333, 1) to (0.666667, 0.666667);
\draw[fill] (0.833333, 1) to (1.33333, 0);
\draw[fill] (0.5, 1) circle (0.045);
\draw[fill] (0.5, 1) to (0.666667, 0.666667);
\draw[fill] (0.5, 1) to (0, 0);
\draw[fill] (2.33333, 0.666667) circle (0.045);
\draw[fill] (2.33333, 0.666667) to (2.5, 0.333333);
\draw[fill] (2.33333, 0.666667) to (2.16667, 0.333333);
\draw[fill] (0.666667, 0.666667) circle (0.045);
\draw[fill] (0.666667, 0.666667) to (0.833333, 0.333333);
\draw[fill] (0.666667, 0.666667) to (0.5, 0.333333);
\draw[fill] (2.5, 0.333333) circle (0.045);
\draw[fill] (2.5, 0.333333) to (2.66667, 0);
\draw[fill] (2.5, 0.333333) to (2.33333, 0);
\draw[fill] (2.16667, 0.333333) circle (0.045);
\draw[fill] (2.16667, 0.333333) to (2.33333, 0);
\draw[fill] (2.16667, 0.333333) to (2, 0);
\draw[fill] (1.5, 0.333333) circle (0.045);
\draw[fill] (1.5, 0.333333) to (1.66667, 0);
\draw[fill] (1.5, 0.333333) to (1.33333, 0);
\draw[fill] (0.833333, 0.333333) circle (0.045);
\draw[fill] (0.833333, 0.333333) to (1, 0);
\draw[fill] (0.833333, 0.333333) to (0.666667, 0);
\draw[fill] (0.5, 0.333333) circle (0.045);
\draw[fill] (0.5, 0.333333) to (0.666667, 0);
\draw[fill] (0.5, 0.333333) to (0.333333, 0);
\draw[fill] (2.66667, 0) circle (0.045);
\draw[fill] (2.33333, 0) circle (0.045);
\draw[fill] (2, 0) circle (0.045);
\draw[fill] (1.66667, 0) circle (0.045);
\draw[fill] (1.33333, 0) circle (0.045);
\draw[fill] (1, 0) circle (0.045);
\draw[fill] (0.666667, 0) circle (0.045);
\draw[fill] (0.333333, 0) circle (0.045);
\draw[fill] (0, 0) circle (0.045);
\end{tikzpicture}
\hspace{1.2cm}
\begin{tikzpicture}[scale = 3]
\draw[thick] (0, 0) to (1, 0);
\draw[ultra thick] (0.0842603, 0.277778)
	to[out = 0, in = 90, distance = 4.18879] (0.333333, 0)
	to[out = 90, in = 90, distance = 6.28318] (0.833333, 0)
	to[out = 90, in = 90, distance = 4.18879] (0.5, 0)
	to[out = -90, in = -90, distance = 4.18879] (0.166667, 0)
	to[out = -90, in = -90, distance = 6.28318] (0.666667, 0)
	to[out = -90, in = 180, distance = 4.18879] (0.91574, -0.277778);
\draw[help lines] (0.5, 0) circle (0.5);
\draw[fill] (0.0842603, 0.277778) circle (0.0166667);
\draw[fill] (0.91574, -0.277778) circle (0.0166667);
\draw[fill] (0, 0) circle (0.0166667);
\draw[fill] (1, 0) circle (0.0166667);
\end{tikzpicture}
\begin{tikzpicture}[scale = 1]
\draw[fill] (1.33333, 2.66667) circle (0.045);
\draw[fill] (0, 0) circle (0.045);
\draw[fill] (0.66667, 0) circle (0.045);
\draw[fill] (1.33333, 0) circle (0.045);
\draw[fill] (2, 0) circle (0.045);
\draw[fill] (2.66667, 0) circle (0.045);

\draw[fill] (1.33333, 2.66667) to (0, 0);
\draw[fill] (1.33333, 2.66667) to (0.66667, 0);
\draw[fill] (1.33333, 2.66667) to (1.33333, 0);
\draw[fill] (1.33333, 2.66667) to (2, 0);
\draw[fill] (1.33333, 2.66667) to (2.66667, 0);
\end{tikzpicture}
\caption{Examples of meanders and the Hasse diagrams of their posets of submeanders.}
\label{fig: example of graphs}
\end{figure}


\begin{definition} \label{def: insert}
Let $$M = (D, (p_1, p_2,p_3,p_4), (m, l))$$ and $$M' = (D', (p_1', p_2',p_3',p_4'),(m', l'))$$ be two meanders of order $(n, k)$ and $(n',k')$ respectively, and let 
$$M'' = (D'', (p_1'', p_2'',p_3'',p_4''), (m'', l''))$$
be a submeander of $M$  of order $(n'',k'')$ such that $n' \equiv n''\ \mathrm{mod}\ 2$. 
Consider a map $f:\partial D'' \to \partial D'$ such that $f({p}_i'')=p'_i$ for each $i=1,\dots,4$. 
There is a well-defined meander $$\tilde{M} = (\tilde{D}, (p_1, p_2, p_3, p_4), (\tilde{m}, \tilde{l}))$$ where  
\begin{itemize}
    \item $\tilde{D} = \big(D\setminus \operatorname{Int}(D'')\big) \cup_f D'$;
    \item $\tilde{m} = \big(m\setminus \operatorname{Int}(D''\cap m)\big)\cup_f m'$;
    \item $\tilde{l} = \big(l\setminus \operatorname{Int}(D''\cap l)\big)\cup_f l'$.
\end{itemize}
We say that $\tilde{M}$ is obtained by the \emph{insertion of $M'$ into $M$ at $M''$}. 
If the total order of $M'$ is one, we say that $\tilde{M}$ is obtained by the \emph{cut of $M'$ from $M$}.
\end{definition}

\begin{remark}
Let $M$ be a meander of total order $n$, and let $M'$ be a submeander of $M$ of total order $n'>1$. If we cut $M'$ from $M$, we get a meander $M''$ of total order $n-n'+1$. An example of two consecutive cuts is shown in Figure~\ref{fig: example of cuts}.
\end{remark}
\begin{figure}[h]
\begin{tikzpicture}[scale = 1.1]
\draw[thick] (0, 0) to (3, 0);
\draw[ultra thick] (0.0857864, 0.5) to[out = 0, in = 90, distance = 26.3894] (2.1, 0)
 to[out = -90, in = -90, distance = 3.76991] (2.4, 0)
 to[out = 90, in = 90, distance = 3.76991] (2.7, 0)
 to[out = -90, in = -90, distance = 11.3097] (1.8, 0)
 to[out = 90, in = 90, distance = 3.76991] (1.5, 0)
 to[out = -90, in = -90, distance = 11.3097] (0.6, 0)
 to[out = 90, in = 90, distance = 3.76991] (0.9, 0)
 to[out = -90, in = -90, distance = 3.76991] (1.2, 0)
 to[out = 90, in = 90, distance = 11.3097] (0.3, 0)
 to[out = -90, in = 180, distance = 33.9292] (2.91421, -0.5, 0);
\draw[help lines] (1.5, 0) circle (1.5);
\draw[fill] (0.0857864, 0.5) circle (0.05);
\draw[fill] (2.91421, -0.5) circle (0.05);
\draw[fill] (0, 0) circle (0.05);
\draw[fill] (3, 0) circle (0.05);

\draw[fill, gray, opacity = 0.5] (0.9, 0) circle (0.45);
\end{tikzpicture}
\begin{tikzpicture}
\draw [->, ultra thick] (0, 0) to (0.7, 0);
\draw [white, fill] (-0.3,-1.65) circle (0.01);
\draw [white, fill] (1,1.65) circle (0.01);
\end{tikzpicture}
\begin{tikzpicture}[scale = 1.1]
\draw[thick] (0, 0) to (3, 0);
\draw[ultra thick] (0.0476312, 0.375) to[out = 0, in = 90, distance = 23.5619] (1.875, 0)
 to[out = -90, in = -90, distance = 4.71239] (2.25, 0)
 to[out = 90, in = 90, distance = 4.71239] (2.625, 0)
 to[out = -90, in = -90, distance = 14.1372] (1.5, 0)
 to[out = 90, in = 90, distance = 4.71239] (1.125, 0)
 to[out = -90, in = -90, distance = 4.71239] (0.75, 0)
 to[out = 90, in = 90, distance = 4.71239] (0.375, 0)
to[out = -90, in = 180, distance = 32.9867] (2.86359, -0.625, 0);
\draw[help lines] (1.5, 0) circle (1.5);
\draw[fill] (0.0476312, 0.375) circle (0.05);
\draw[fill] (2.86359, -0.625) circle (0.05);
\draw[fill] (0, 0) circle (0.05);
\draw[fill] (3, 0) circle (0.05);

\draw[thick, gray, fill, opacity = 0.5] (2.25, 0) circle (0.375*1.5);
\end{tikzpicture}
\begin{tikzpicture}
\draw [->, ultra thick] (0, 0) to (0.7, 0);
\draw [white, fill] (-0.3,-1.65) circle (0.01);
\draw [white, fill] (1,1.65) circle (0.01);
\end{tikzpicture}
\begin{tikzpicture}[scale = 1.1]
\draw[thick] (0, 0) to (3, 0);
\draw[ultra thick] (0.0857864, 0.5) to[out = 0, in = 90, distance = 31.4159] (2.5, 0)
 to[out = -90, in = -90, distance = 6.28318] (2, 0)
 to[out = 90, in = 90, distance = 6.28318] (1.5, 0)
 to[out = -90, in = -90, distance = 6.28318] (1, 0)
 to[out = 90, in = 90, distance = 6.28318] (0.5, 0)
to[out = -90, in = 180, distance = 31.4159] (2.91421, -0.5, 0);
\draw[help lines] (1.5, 0) circle (1.5);
\draw[fill] (0.0857864, 0.5) circle (0.05);
\draw[fill] (2.91421, -0.5) circle (0.05);
\draw[fill] (0, 0) circle (0.05);
\draw[fill] (3, 0) circle (0.05);
\end{tikzpicture}
\caption{Example of two consecutive cuts.}
\label{fig: example of cuts}
\end{figure}

\subsection{The structure of a 2-colored operad}\label{sec: operad}
The insertion of one meander into another gives rise to the operations on the set of all equivalence classes of meanders.
Let $M$ be a meander of order $(n, k)$, let $t_1<t_2<\dots < t_n$ (resp. $s_1<s_2<\dots<s_k$) be the labels of the transverse (resp. non-transverse) intersections of~$M$. Then for a natural number $i$, let $M|_i$ (resp. $M|_{(i)}$) be a submeander of $M$ with the only intersection with the label $t_i$ (resp. $s_i$). 

Now for an arbitrary meander $M'$ of order $(2n'+1, k')$ we can define $M\circ_i M'$ to be a meander of order $(n+2n', k+k')$ obtained by the insertion of $M'$ into $M$ at $M|_i$. Note that up to equivalence  $M\circ_i M'$ is well-defined. 
Analogously, we can define $M \bullet_i M'$ to be the result of the insertion of a meander $M'$ of order $(2n', k')$  into $M$ at $M|_{(i)}$. 

The straightforward check shows that these operations form a 2-colored operad on the set of equivalence classes of meanders (for the definition of a colored operad, see~\cite{LV12, Y16}, and for the examples of applications of operads in combinatorics see~\cite{G18}). 
\begin{thm}\label{thm: operad}
The set $\Mset = \bigcup\limits_{n\geq 0, k \geq 0} \Mset_{n,k}$ together with the set of operations
\begin{align*}
    &\circ_i: \Mset_{n,k} \times \Mset_{2n'+1, k'} \to \Mset_{n+2n', k+k'} &  n \geq 1,\ k \geq 0,\ 1\leq i\leq n,\\
    &\bullet_i: \Mset_{n,k} \times \Mset_{2n', k'} \to \Mset_{n+2n', k+k'-1}  &n \geq 0,\ k \geq 1,\   1\leq i \leq k.
\end{align*}
form a 2-colored operad.
\end{thm}

\section{Decomposition} \label{sec: decomposition}
In this section, we define prime meanders and show that each meander can be canonically decomposed into prime components. 

\subsection{Preliminary lemmas}
\begin{lemma} \label{lem: submeanders and permutations}
    Let 
    $$M = (D, (p_1, p_2,p_3,p_4), (m, l))$$ 
    be a meander of total order $N$ with permutation $(\alpha_1, \alpha_2, \dots, \alpha_N)$, and let $A = \{\alpha_u, \alpha_{u+1}, \dots, \alpha_{u+v}\}$ be some subset of its labels (here $u$ and $v$ are natural numbers). Then there exists
    $$M' = (D', (p_1', p_2',p_3',p_4'),(m', l'))$$ 
    a submeander of $M$ containing exactly the intersections with labels from $A$ if and only if 
    $$
    \max_{\alpha \in A} \alpha - 
    \min_{\alpha \in A} \alpha = v.
    $$
    In other words, if and only if $A$ consists of consecutive numbers. 
\end{lemma}
\begin{proof}
    Let $M'$ be a submeander of $M$ as above, and suppose $$
    \max_{\alpha \in A} \alpha -
    \min_{\alpha \in A} \alpha \neq v.$$ Note that this difference cannot be less than $v$ (as there are $v+1$ different elements in $A$), hence it must be greater than $v$. In this case there exists an intersection with label $\tilde{\alpha}$ such that $\tilde{\alpha} \notin A$ and
    $$
    \min_{\alpha \in A} \alpha < 
    \tilde{\alpha} < 
    \max_{\alpha \in A} \alpha.$$ 
    Consequently, $D'$ does not contain this point, resulting in $m' = D' \cap m$ being disconnected, which leads to a contradiction.

    Conversely, assume $A$ consists of labels as specified, with $A$ being a set of consecutive labels. $M'$ can be constructed explicitly. First of all, note that there exists $l' \subset l$ (resp. $m' \subset m$) --- a connected subset of $l$ (resp. $m$) containing precisely the intersections of $M$ with the labels from $A$. All that remains is to select an arbitrary disk $D' \subset D$ such that $D\cap l = l'$ and $D\cap m = m'$. 
\end{proof}

For a given meander $M$ of order $(n,k)$ and of total order greater than one, the cardinality of $\Submnd(M)$ cannot be less than $n+k+1$, since there are always $n+k$ submeanders with a single intersection, and there is also a submeander equivalent to $M$. Conversely, the cardinality of $\Submnd(M)$ cannot exceed $\frac{(n+k)(n+k+1)}{2}$, since each submeander can only contain intersections with consecutive labels (see Lemma~\ref{lem: submeanders and permutations}).

\begin{definition}
Let $M$ be a meander of order $(n,k)$. $M$ is said to be \emph{irreducible} if its total order is more than two and $|\Submnd(M)| = n+k+1$. 
$M$ is said to be a \emph{snake} if its total order is more than one and  $|\Submnd(M)| =\frac{(n+k)(n+k+1)}{2}$.\\
A meander is called \emph{prime} if it is either a snake or an irreducible meander. 
\end{definition}
Prime meanders are the building blocks from which any meander can be constructed. 
The right meander in Figure~\ref{fig: example of graphs} is irreducible, while the left one is non-prime. 
Examples of snakes are shown in Figure~\ref{fig: example of snake}.

\begin{remark}
    In~\cite{LZ92} a definition of an irreducible meander system was introduced. This definition is different from ours but they are similar. To be more precise, a \emph{meander system} is a triple: $(D, \{m_1, \dots, m_r\}, l)$, where $D$ is a Euclidean 2-dimensional disk, $l$ is an image of smooth proper embedding of a segment into $D$, and $m_1,\dots, m_r$ (here $r\geq 1$) are pairwise disjoint images of smooth embeddings of a circle into $D$ that intersects $l$ only transversely. Two meander systems are called equivalent if they are homeomorphic as triples. A meander system $(D', \{m_1', \dots, m_{r'}'\}, l')$ is a subsystem of $(D, \{m_1, \dots, m_r\}, l)$ if (i) $D'\subseteq D$, (ii) $ \{m_1', \dots, m_{r'}'\} = D' \cap \{m_1, \dots, m_r\}$, (iii) $l' = D' \cap l$. Finally, a meander system $M$ is called \emph{irreducible} if all its subsystems are equivalent to $M$. 
\end{remark}

\begin{figure}[h]
\begin{tikzpicture}
\draw[thick] (0, 0) to (3, 0);
\draw[ultra thick] (0.0303062, 0.3) to[out = 0, in = 90, distance = 3.76991] (0.3, 0)
 to[out = -90, in = -90, distance = 3.76991] (0.6, 0)
 to[out = 90, in = 90, distance = 3.76991] (0.9, 0)
 to[out = -90, in = -90, distance = 3.76991] (1.2, 0)
 to[out = 90, in = 90, distance = 3.76991] (1.5, 0)
 to[out = -90, in = -90, distance = 3.76991] (1.8, 0)
 to[out = 90, in = 90, distance = 3.76991] (2.1, 0)
 to[out = -90, in = -90, distance = 3.76991] (2.4, 0)
 to[out = 90, in = 90, distance = 3.76991] (2.7, 0)
to[out = -90, in = 180, distance = 3.76991] (2.96969, -0.3, 0);
\draw[help lines] (1.5, 0) circle (1.5);
\draw[fill] (0.0303062, 0.3) circle (0.05);
\draw[fill] (2.96969, -0.3) circle (0.05);
\draw[fill] (0, 0) circle (0.05);
\draw[fill] (3, 0) circle (0.05);
\end{tikzpicture}
\begin{tikzpicture}[scale = 1]
\draw[fill] (1.33333, 2.66667) circle (0.045);
\draw[fill] (1.33333, 2.66667) to (1.5, 2.33333);
\draw[fill] (1.33333, 2.66667) to (1.16667, 2.33333);
\draw[fill] (1.5, 2.33333) circle (0.045);
\draw[fill] (1.5, 2.33333) to (1.66667, 2);
\draw[fill] (1.5, 2.33333) to (1.33333, 2);
\draw[fill] (1.16667, 2.33333) circle (0.045);
\draw[fill] (1.16667, 2.33333) to (1.33333, 2);
\draw[fill] (1.16667, 2.33333) to (1, 2);
\draw[fill] (1.66667, 2) circle (0.045);
\draw[fill] (1.66667, 2) to (1.83333, 1.66667);
\draw[fill] (1.66667, 2) to (1.5, 1.66667);
\draw[fill] (1.33333, 2) circle (0.045);
\draw[fill] (1.33333, 2) to (1.5, 1.66667);
\draw[fill] (1.33333, 2) to (1.16667, 1.66667);
\draw[fill] (1, 2) circle (0.045);
\draw[fill] (1, 2) to (1.16667, 1.66667);
\draw[fill] (1, 2) to (0.833333, 1.66667);
\draw[fill] (1.83333, 1.66667) circle (0.045);
\draw[fill] (1.83333, 1.66667) to (2, 1.33333);
\draw[fill] (1.83333, 1.66667) to (1.66667, 1.33333);
\draw[fill] (1.5, 1.66667) circle (0.045);
\draw[fill] (1.5, 1.66667) to (1.66667, 1.33333);
\draw[fill] (1.5, 1.66667) to (1.33333, 1.33333);
\draw[fill] (1.16667, 1.66667) circle (0.045);
\draw[fill] (1.16667, 1.66667) to (1.33333, 1.33333);
\draw[fill] (1.16667, 1.66667) to (1, 1.33333);
\draw[fill] (0.833333, 1.66667) circle (0.045);
\draw[fill] (0.833333, 1.66667) to (1, 1.33333);
\draw[fill] (0.833333, 1.66667) to (0.666667, 1.33333);
\draw[fill] (2, 1.33333) circle (0.045);
\draw[fill] (2, 1.33333) to (2.16667, 1);
\draw[fill] (2, 1.33333) to (1.83333, 1);
\draw[fill] (1.66667, 1.33333) circle (0.045);
\draw[fill] (1.66667, 1.33333) to (1.83333, 1);
\draw[fill] (1.66667, 1.33333) to (1.5, 1);
\draw[fill] (1.33333, 1.33333) circle (0.045);
\draw[fill] (1.33333, 1.33333) to (1.5, 1);
\draw[fill] (1.33333, 1.33333) to (1.16667, 1);
\draw[fill] (1, 1.33333) circle (0.045);
\draw[fill] (1, 1.33333) to (1.16667, 1);
\draw[fill] (1, 1.33333) to (0.833333, 1);
\draw[fill] (0.666667, 1.33333) circle (0.045);
\draw[fill] (0.666667, 1.33333) to (0.833333, 1);
\draw[fill] (0.666667, 1.33333) to (0.5, 1);
\draw[fill] (2.16667, 1) circle (0.045);
\draw[fill] (2.16667, 1) to (2.33333, 0.666667);
\draw[fill] (2.16667, 1) to (2, 0.666667);
\draw[fill] (1.83333, 1) circle (0.045);
\draw[fill] (1.83333, 1) to (2, 0.666667);
\draw[fill] (1.83333, 1) to (1.66667, 0.666667);
\draw[fill] (1.5, 1) circle (0.045);
\draw[fill] (1.5, 1) to (1.66667, 0.666667);
\draw[fill] (1.5, 1) to (1.33333, 0.666667);
\draw[fill] (1.16667, 1) circle (0.045);
\draw[fill] (1.16667, 1) to (1.33333, 0.666667);
\draw[fill] (1.16667, 1) to (1, 0.666667);
\draw[fill] (0.833333, 1) circle (0.045);
\draw[fill] (0.833333, 1) to (1, 0.666667);
\draw[fill] (0.833333, 1) to (0.666667, 0.666667);
\draw[fill] (0.5, 1) circle (0.045);
\draw[fill] (0.5, 1) to (0.666667, 0.666667);
\draw[fill] (0.5, 1) to (0.333333, 0.666667);
\draw[fill] (2.33333, 0.666667) circle (0.045);
\draw[fill] (2.33333, 0.666667) to (2.5, 0.333333);
\draw[fill] (2.33333, 0.666667) to (2.16667, 0.333333);
\draw[fill] (2, 0.666667) circle (0.045);
\draw[fill] (2, 0.666667) to (2.16667, 0.333333);
\draw[fill] (2, 0.666667) to (1.83333, 0.333333);
\draw[fill] (1.66667, 0.666667) circle (0.045);
\draw[fill] (1.66667, 0.666667) to (1.83333, 0.333333);
\draw[fill] (1.66667, 0.666667) to (1.5, 0.333333);
\draw[fill] (1.33333, 0.666667) circle (0.045);
\draw[fill] (1.33333, 0.666667) to (1.5, 0.333333);
\draw[fill] (1.33333, 0.666667) to (1.16667, 0.333333);
\draw[fill] (1, 0.666667) circle (0.045);
\draw[fill] (1, 0.666667) to (1.16667, 0.333333);
\draw[fill] (1, 0.666667) to (0.833333, 0.333333);
\draw[fill] (0.666667, 0.666667) circle (0.045);
\draw[fill] (0.666667, 0.666667) to (0.833333, 0.333333);
\draw[fill] (0.666667, 0.666667) to (0.5, 0.333333);
\draw[fill] (0.333333, 0.666667) circle (0.045);
\draw[fill] (0.333333, 0.666667) to (0.5, 0.333333);
\draw[fill] (0.333333, 0.666667) to (0.166667, 0.333333);
\draw[fill] (2.5, 0.333333) circle (0.045);
\draw[fill] (2.5, 0.333333) to (2.66667, 0);
\draw[fill] (2.5, 0.333333) to (2.33333, 0);
\draw[fill] (2.16667, 0.333333) circle (0.045);
\draw[fill] (2.16667, 0.333333) to (2.33333, 0);
\draw[fill] (2.16667, 0.333333) to (2, 0);
\draw[fill] (1.83333, 0.333333) circle (0.045);
\draw[fill] (1.83333, 0.333333) to (2, 0);
\draw[fill] (1.83333, 0.333333) to (1.66667, 0);
\draw[fill] (1.5, 0.333333) circle (0.045);
\draw[fill] (1.5, 0.333333) to (1.66667, 0);
\draw[fill] (1.5, 0.333333) to (1.33333, 0);
\draw[fill] (1.16667, 0.333333) circle (0.045);
\draw[fill] (1.16667, 0.333333) to (1.33333, 0);
\draw[fill] (1.16667, 0.333333) to (1, 0);
\draw[fill] (0.833333, 0.333333) circle (0.045);
\draw[fill] (0.833333, 0.333333) to (1, 0);
\draw[fill] (0.833333, 0.333333) to (0.666667, 0);
\draw[fill] (0.5, 0.333333) circle (0.045);
\draw[fill] (0.5, 0.333333) to (0.666667, 0);
\draw[fill] (0.5, 0.333333) to (0.333333, 0);
\draw[fill] (0.166667, 0.333333) circle (0.045);
\draw[fill] (0.166667, 0.333333) to (0.333333, 0);
\draw[fill] (0.166667, 0.333333) to (0, 0);
\draw[fill] (2.66667, 0) circle (0.045);
\draw[fill] (2.33333, 0) circle (0.045);
\draw[fill] (2, 0) circle (0.045);
\draw[fill] (1.66667, 0) circle (0.045);
\draw[fill] (1.33333, 0) circle (0.045);
\draw[fill] (1, 0) circle (0.045);
\draw[fill] (0.666667, 0) circle (0.045);
\draw[fill] (0.333333, 0) circle (0.045);
\draw[fill] (0, 0) circle (0.045);
\end{tikzpicture}
\hspace{1.2cm}
\begin{tikzpicture}
\draw[thick] (0, 0) to (3, 0);
\draw[ultra thick] (0.0303062, 0.3) to[out = 0, in = 90, distance = 33.9292] (2.7, 0)
 to[out = -90, in = -90, distance = 3.76991] (2.4, 0)
 to[out = 90, in = 90, distance = 3.76991] (2.1, 0)
 to[out = -90, in = -90, distance = 3.76991] (1.8, 0)
 to[out = 90, in = 90, distance = 3.76991] (1.5, 0)
 to[out = -90, in = -90, distance = 3.76991] (1.2, 0)
 to[out = 90, in = 90, distance = 3.76991] (0.9, 0)
 to[out = -90, in = -90, distance = 3.76991] (0.6, 0)
 to[out = 90, in = 90, distance = 3.76991] (0.3, 0)
to[out = -90, in = 180, distance = 33.9292] (2.96969, -0.3, 0);
\draw[help lines] (1.5, 0) circle (1.5);
\draw[fill] (0.0303062, 0.3) circle (0.05);
\draw[fill] (2.96969, -0.3) circle (0.05);
\draw[fill] (0, 0) circle (0.05);
\draw[fill] (3, 0) circle (0.05);
\end{tikzpicture}
\begin{tikzpicture}[scale = 1]
\draw[fill] (1.33333, 2.66667) circle (0.045);
\draw[fill] (1.33333, 2.66667) to (1.5, 2.33333);
\draw[fill] (1.33333, 2.66667) to (1.16667, 2.33333);
\draw[fill] (1.5, 2.33333) circle (0.045);
\draw[fill] (1.5, 2.33333) to (1.66667, 2);
\draw[fill] (1.5, 2.33333) to (1.33333, 2);
\draw[fill] (1.16667, 2.33333) circle (0.045);
\draw[fill] (1.16667, 2.33333) to (1.33333, 2);
\draw[fill] (1.16667, 2.33333) to (1, 2);
\draw[fill] (1.66667, 2) circle (0.045);
\draw[fill] (1.66667, 2) to (1.83333, 1.66667);
\draw[fill] (1.66667, 2) to (1.5, 1.66667);
\draw[fill] (1.33333, 2) circle (0.045);
\draw[fill] (1.33333, 2) to (1.5, 1.66667);
\draw[fill] (1.33333, 2) to (1.16667, 1.66667);
\draw[fill] (1, 2) circle (0.045);
\draw[fill] (1, 2) to (1.16667, 1.66667);
\draw[fill] (1, 2) to (0.833333, 1.66667);
\draw[fill] (1.83333, 1.66667) circle (0.045);
\draw[fill] (1.83333, 1.66667) to (2, 1.33333);
\draw[fill] (1.83333, 1.66667) to (1.66667, 1.33333);
\draw[fill] (1.5, 1.66667) circle (0.045);
\draw[fill] (1.5, 1.66667) to (1.66667, 1.33333);
\draw[fill] (1.5, 1.66667) to (1.33333, 1.33333);
\draw[fill] (1.16667, 1.66667) circle (0.045);
\draw[fill] (1.16667, 1.66667) to (1.33333, 1.33333);
\draw[fill] (1.16667, 1.66667) to (1, 1.33333);
\draw[fill] (0.833333, 1.66667) circle (0.045);
\draw[fill] (0.833333, 1.66667) to (1, 1.33333);
\draw[fill] (0.833333, 1.66667) to (0.666667, 1.33333);
\draw[fill] (2, 1.33333) circle (0.045);
\draw[fill] (2, 1.33333) to (2.16667, 1);
\draw[fill] (2, 1.33333) to (1.83333, 1);
\draw[fill] (1.66667, 1.33333) circle (0.045);
\draw[fill] (1.66667, 1.33333) to (1.83333, 1);
\draw[fill] (1.66667, 1.33333) to (1.5, 1);
\draw[fill] (1.33333, 1.33333) circle (0.045);
\draw[fill] (1.33333, 1.33333) to (1.5, 1);
\draw[fill] (1.33333, 1.33333) to (1.16667, 1);
\draw[fill] (1, 1.33333) circle (0.045);
\draw[fill] (1, 1.33333) to (1.16667, 1);
\draw[fill] (1, 1.33333) to (0.833333, 1);
\draw[fill] (0.666667, 1.33333) circle (0.045);
\draw[fill] (0.666667, 1.33333) to (0.833333, 1);
\draw[fill] (0.666667, 1.33333) to (0.5, 1);
\draw[fill] (2.16667, 1) circle (0.045);
\draw[fill] (2.16667, 1) to (2.33333, 0.666667);
\draw[fill] (2.16667, 1) to (2, 0.666667);
\draw[fill] (1.83333, 1) circle (0.045);
\draw[fill] (1.83333, 1) to (2, 0.666667);
\draw[fill] (1.83333, 1) to (1.66667, 0.666667);
\draw[fill] (1.5, 1) circle (0.045);
\draw[fill] (1.5, 1) to (1.66667, 0.666667);
\draw[fill] (1.5, 1) to (1.33333, 0.666667);
\draw[fill] (1.16667, 1) circle (0.045);
\draw[fill] (1.16667, 1) to (1.33333, 0.666667);
\draw[fill] (1.16667, 1) to (1, 0.666667);
\draw[fill] (0.833333, 1) circle (0.045);
\draw[fill] (0.833333, 1) to (1, 0.666667);
\draw[fill] (0.833333, 1) to (0.666667, 0.666667);
\draw[fill] (0.5, 1) circle (0.045);
\draw[fill] (0.5, 1) to (0.666667, 0.666667);
\draw[fill] (0.5, 1) to (0.333333, 0.666667);
\draw[fill] (2.33333, 0.666667) circle (0.045);
\draw[fill] (2.33333, 0.666667) to (2.5, 0.333333);
\draw[fill] (2.33333, 0.666667) to (2.16667, 0.333333);
\draw[fill] (2, 0.666667) circle (0.045);
\draw[fill] (2, 0.666667) to (2.16667, 0.333333);
\draw[fill] (2, 0.666667) to (1.83333, 0.333333);
\draw[fill] (1.66667, 0.666667) circle (0.045);
\draw[fill] (1.66667, 0.666667) to (1.83333, 0.333333);
\draw[fill] (1.66667, 0.666667) to (1.5, 0.333333);
\draw[fill] (1.33333, 0.666667) circle (0.045);
\draw[fill] (1.33333, 0.666667) to (1.5, 0.333333);
\draw[fill] (1.33333, 0.666667) to (1.16667, 0.333333);
\draw[fill] (1, 0.666667) circle (0.045);
\draw[fill] (1, 0.666667) to (1.16667, 0.333333);
\draw[fill] (1, 0.666667) to (0.833333, 0.333333);
\draw[fill] (0.666667, 0.666667) circle (0.045);
\draw[fill] (0.666667, 0.666667) to (0.833333, 0.333333);
\draw[fill] (0.666667, 0.666667) to (0.5, 0.333333);
\draw[fill] (0.333333, 0.666667) circle (0.045);
\draw[fill] (0.333333, 0.666667) to (0.5, 0.333333);
\draw[fill] (0.333333, 0.666667) to (0.166667, 0.333333);
\draw[fill] (2.5, 0.333333) circle (0.045);
\draw[fill] (2.5, 0.333333) to (2.66667, 0);
\draw[fill] (2.5, 0.333333) to (2.33333, 0);
\draw[fill] (2.16667, 0.333333) circle (0.045);
\draw[fill] (2.16667, 0.333333) to (2.33333, 0);
\draw[fill] (2.16667, 0.333333) to (2, 0);
\draw[fill] (1.83333, 0.333333) circle (0.045);
\draw[fill] (1.83333, 0.333333) to (2, 0);
\draw[fill] (1.83333, 0.333333) to (1.66667, 0);
\draw[fill] (1.5, 0.333333) circle (0.045);
\draw[fill] (1.5, 0.333333) to (1.66667, 0);
\draw[fill] (1.5, 0.333333) to (1.33333, 0);
\draw[fill] (1.16667, 0.333333) circle (0.045);
\draw[fill] (1.16667, 0.333333) to (1.33333, 0);
\draw[fill] (1.16667, 0.333333) to (1, 0);
\draw[fill] (0.833333, 0.333333) circle (0.045);
\draw[fill] (0.833333, 0.333333) to (1, 0);
\draw[fill] (0.833333, 0.333333) to (0.666667, 0);
\draw[fill] (0.5, 0.333333) circle (0.045);
\draw[fill] (0.5, 0.333333) to (0.666667, 0);
\draw[fill] (0.5, 0.333333) to (0.333333, 0);
\draw[fill] (0.166667, 0.333333) circle (0.045);
\draw[fill] (0.166667, 0.333333) to (0.333333, 0);
\draw[fill] (0.166667, 0.333333) to (0, 0);
\draw[fill] (2.66667, 0) circle (0.045);
\draw[fill] (2.33333, 0) circle (0.045);
\draw[fill] (2, 0) circle (0.045);
\draw[fill] (1.66667, 0) circle (0.045);
\draw[fill] (1.33333, 0) circle (0.045);
\draw[fill] (1, 0) circle (0.045);
\draw[fill] (0.666667, 0) circle (0.045);
\draw[fill] (0.333333, 0) circle (0.045);
\draw[fill] (0, 0) circle (0.045);
\end{tikzpicture}
\caption{Examples of snakes.}
\label{fig: example of snake}
\end{figure}

\begin{lemma} \label{lem: prime submeander exist}
    If $M$ is a non-prime meander of total order greater than one, then there exists prime meander $M'$ that is a submeander of $M$.
\end{lemma}
\begin{proof}
    Let $M$ be a non-prime meander of total order greater than one. Let us choose a submeander $M'$ of $M$ with the minimal total order greater than one (such submeander exists because $M$ is non-prime). Let the total order of $M'$ be $k$. If $k = 2$, then it is a snake (since all meanders of total order two are snakes); otherwise, the absence of any submeander of $M'$ with a total order between one and $k$ implies that $M'$ is irreducible.
\end{proof}

\begin{lemma}\label{lem: snakes classification}
A meander of total order $n$ is a snake if and only if its permutation is either $(1,2,\dots,n)$ or $(n, n~-~1, \dots, 1)$. 
\end{lemma}
\begin{proof}
    This follows from Lemma~\ref{lem: submeanders and permutations}. Indeed, if a meander with permutation $(\alpha_1, \dots, \alpha_n)$ is a snake, then for each $1\leq i<n$ we have $|\alpha_i - \alpha_{i+1}| = 1$. 
\end{proof}

\begin{definition}\label{def: direct and inverse snakes}
    A snake~$M$ of total order $n$ is said to be \emph{direct} if its permutation is $(1,2,\dots,n)$. Otherwise, $M$ is said to be \emph{inverse}.
\end{definition}

\begin{definition}
    Let $M$ be a meander, and let $M'$ be its submeander that is a snake. We say that $M'$ is a \emph{maximal snake in $M$} if for each snake $M''$ that is a submeander of $M$ if $M' \leq M''$ then $M'$ is equivalent to $M''$. 
\end{definition}

\begin{lemma} \label{lem: snakes dont intersect}
    Let $M$ be a meander and let $M_1$ and $M_2$ be two maximal snakes in $M$. If $M_1$ and $M_2$ are not equivalent with respect to $M$, then no submeander of $M$ is both a submeander of $M_1$ and a submeander of $M_2$.
\end{lemma}
\begin{proof}
    The statement follows from Lemma~\ref{lem: snakes classification}. If both $M_1$ and $M_2$ are direct snakes, the statement is clear. Let $M_1$ be an inverse snake, and suppose that $M_1$ and $M_2$ have a common submeander. Then there exists $M_2'$ --- a submeander of $M_2$ such that (i) $M_2'$ is equivalent to $M_2$, (ii) $M_2'$ is a submeander of $M_1$. In this case, $M_2$ would not be a maximal snake, leading to a contradiction.
\end{proof}

\begin{lemma}\label{lem: irreducible dont intersect}
    Let $M$ be a meander, and let $M_1$ and $M_2$ be two irreducible meanders that are submeanders of $M$. If $M_1$ and $M_2$ are not equivalent with respect to $M$, then no submeander of $M$ is both a submeander of $M_1$ and a submeander of $M_2$.
\end{lemma}
\begin{proof}
    Let $M'$ be both a submeander of $M_1$ and a submeander of $M_2$. The total order of $M'$ must be one (since $M_1$ and $M_2$ are irreducible and not equivalent with respect to $M$). Let $(\alpha_{i_1},\dots, \alpha_{i_r})$ be the permutation of $M_1$, and let $(\alpha')$ be the permutation of $M'$. Note that $\alpha'$ is either $\min\limits_{j=1,\dots, r} \alpha_{i_j}$ or $\max\limits_{j=1,\dots, r} \alpha_{i_j}$ (otherwise $M'$ is not a submeander of $M_2$). But in this case 
    $$
        \max_{\substack{j=1,\dots, r;\\ \alpha_{i_j} \neq \alpha'}} \alpha_{i_j} - \min_{\substack{j=1,\dots, r;\\ \alpha_{i_j} \neq \alpha'}} \alpha_{i_j} = r-2,
    $$
    and from Lemma~\ref{lem: submeanders and permutations} it follows that $M_1$ is not irreducible (here we also use that the total order of $M_1$ is greater than two).
\end{proof}

\subsection{Description of the factorization} \label{subsec: decomposition}
Now that we have established the foundational lemmas, we are ready to define the factorization process for an arbitrary meander. Let $M$ be a meander of total order $N > 1$, and let $\operatorname{P}(M)$ be the set of all maximal snakes and irreducible submeanders in $M$ that are not equivalent to each other with respect to $M$ (due to Lemma~\ref{lem: prime submeander exist} this set is not empty). 
Now, if we cut each $M' \in \operatorname{P}(M)$ from $M$, we obtain a meander $M_1$ of total order less than $N$ (Lemma~\ref{lem: snakes dont intersect} and Lemma~\ref{lem: irreducible dont intersect} ensure that $M_1$ is well-defined). 
If the total order of $M_1$ is greater than one, we can repeat this procedure. Thus, we obtain a finite sequence of meanders that ends with a meander of total order one.

\begin{example}
    Let us consider an example of the decomposition. In Figure~\ref{fig: example of decomposition 1}(a) we present a non-prime meander $M$ with the highlighted set $\operatorname{P}(M)$. Figure~\ref{fig: example of decomposition 1}(b) shows $M_1$ --- the result of the cut. Figure~\ref{fig: example of decomposition 1}(c) shows $M_1$ with the highlighted set $\operatorname{P}(M_1)$. Finally, Figure~\ref{fig: example of decomposition 1}(d) shows $M_2$, which turns out to be an irreducible meander, so $\operatorname{P}(M_2)=\{M_2\}$, and $M_3$ has order $(1,0)$ (so we do not provide a separate picture for it).
\end{example}

\begin{figure}[h]
    \centering
    \begin{tikzpicture}[scale = 5]    
    \node (a) at (0.5, -0.6) {(a)};
\draw[thick] (0, 0) to (1, 0);
\draw[ultra thick] (0.0285955, 0.166667)
	to[out = 0, in = 90, distance = 6.28318] (0.5, 0)
	to[out = -90, in = -90, distance = 0.571199] (0.545455, 0)
	to[out = 90, in = 90, distance = 5.14079] (0.954545, 0)
	to[out = -90, in = -90, distance = 0.571199] (0.909091, 0)
	to[out = 90, in = 90, distance = 3.99839] (0.590909, 0)
	to[out = -90, in = -90, distance = 0.571199] (0.636364, 0)
	to[out = 90, in = 90, distance = 0.571199] (0.681818, 0)
	to[out = -90, in = -90, distance = 0.571199] (0.727273, 0)
	to[out = 90, in = 90, distance = 0.571199] (0.772727, 0)
	to[out = -90, in = -90, distance = 6.28318] (0.272727, 0)
	to[out = 90, in = 90, distance = 0.571199] (0.318182, 0)
	to[out = -90, in = -90, distance = 1.7136] (0.454545, 0)
	to[out = 90, in = 90, distance = 5.14079] (0.0454545, 0)
	to[out = -90, in = -90, distance = 0.571199] (0.0909091, 0)
	to[out = 90, in = 90, distance = 0.571199] (0.136364, 0)
	to[out = -90, in = -90, distance = 0.571199] (0.181818, 0)
	to[out = 90, in = 90, distance = 2.85599] (0.409091, 0)
	to[out = -90, in = -90, distance = 0.571199] (0.363636, 0)
	to[out = 90, in = 90, distance = 1.7136] (0.227273, 0)
	to[out = -90, in = -90, distance = 7.42558] (0.818182, 0)
	to[out = 90, in = 90, distance = 0.571199] (0.863636, 0)
	to[out = -90, in = 180, distance = 1.7136] (0.945362, -0.227273);
\draw[help lines] (0.5, 0) circle (0.5);
\draw[fill] (0.0285955, 0.166667) circle (0.0166667);
\draw[fill] (0.945362, -0.227273) circle (0.0166667);
\draw[fill] (0, 0) circle (0.0166667);
\draw[fill] (1, 0) circle (0.0166667);

\draw[thick, gray, fill, opacity = 0.5] (2.5/22, 0) circle (1.8/22);
\draw[thick, gray, fill, opacity = 0.5] (6.5/22, 0) circle (0.8/22);
\draw[thick, gray, fill, opacity = 0.5] (8.5/22, 0) circle (0.8/22);
\draw[thick, gray, fill, opacity = 0.5] (11.5/22, 0) circle (0.8/22);
\draw[thick, gray, fill, opacity = 0.5] (15/22, 0) circle (2.3/22);
\draw[thick, gray, fill, opacity = 0.5] (20.5/22, 0) circle (0.8/22);
\draw[thick, gray, fill, opacity = 0.5] (18.5/22, 0) circle (0.8/22);
\end{tikzpicture}
\begin{tikzpicture}[scale = 0.2]
    \begin{scope}
        \clip (-1.5,-16.5) rectangle (5,11.6);
        \draw[ultra thick, ->] (-0.5,0) to (4, 0);
    \end{scope}  
\end{tikzpicture}
\begin{tikzpicture}[scale = 5]
    \node (a) at (0.5, -0.6) {(b)};
\draw[thick] (0, 0) to (1, 0);
\draw[ultra thick] (0.0417424, 0.2)
	to[out = 0, in = 90, distance = 7.53982] (0.6, 0)
	to[out = 90, in = 90, distance = 3.76991] (0.9, 0)
	to[out = 90, in = 90, distance = 2.51327] (0.7, 0)
	to[out = -90, in = -90, distance = 5.02655] (0.3, 0)
	to[out = -90, in = -90, distance = 2.51327] (0.5, 0)
	to[out = 90, in = 90, distance = 5.02655] (0.1, 0)
	to[out = 90, in = 90, distance = 3.76991] (0.4, 0)
	to[out = 90, in = 90, distance = 2.51327] (0.2, 0)
	to[out = -90, in = -90, distance = 7.53982] (0.8, 0)
	to[out = -90, in = 180, distance = 2.51327] (0.922953, -0.266667);
\draw[help lines] (0.5, 0) circle (0.5);
\draw[fill] (0.0417424, 0.2) circle (0.0166667);
\draw[fill] (0.922953, -0.266667) circle (0.0166667);
\draw[fill] (0, 0) circle (0.0166667);
\draw[fill] (1, 0) circle (0.0166667);
\end{tikzpicture}

\begin{tikzpicture}[scale = 5]
    \node (a) at (0.5, -0.6) {(c)};
\draw[thick] (0, 0) to (1, 0);
\draw[ultra thick] (0.0417424, 0.2)
	to[out = 0, in = 90, distance = 7.53982] (0.6, 0)
	to[out = 90, in = 90, distance = 3.76991] (0.9, 0)
	to[out = 90, in = 90, distance = 2.51327] (0.7, 0)
	to[out = -90, in = -90, distance = 5.02655] (0.3, 0)
	to[out = -90, in = -90, distance = 2.51327] (0.5, 0)
	to[out = 90, in = 90, distance = 5.02655] (0.1, 0)
	to[out = 90, in = 90, distance = 3.76991] (0.4, 0)
	to[out = 90, in = 90, distance = 2.51327] (0.2, 0)
	to[out = -90, in = -90, distance = 7.53982] (0.8, 0)
	to[out = -90, in = 180, distance = 2.51327] (0.922953, -0.266667);
\draw[help lines] (0.5, 0) circle (0.5);
\draw[fill] (0.0417424, 0.2) circle (0.0166667);
\draw[fill] (0.922953, -0.266667) circle (0.0166667);
\draw[fill] (0, 0) circle (0.0166667);
\draw[fill] (1, 0) circle (0.0166667);
\draw[thick, gray, fill, opacity = 0.5] (0.5/10, 0)
    to[out = 90, in = 90, distance = 6] (5.5/10, 0)
    to[out = -90, in = -90, distance = 6] (0.5/10, 0);
\end{tikzpicture}
\begin{tikzpicture}[scale = 0.2]
    \begin{scope}
        \clip (-1.5,-16.5) rectangle (5,11.6);
        \draw[ultra thick, ->] (-0.5,0) to (4, 0);
    \end{scope}  
\end{tikzpicture}
\begin{tikzpicture}[scale = 5]
    \node (a) at (0.5, -0.6) {(d)};
\draw[thick] (0, 0) to (1, 0);
\draw[ultra thick] (0.0842603, 0.277778)
	to[out = 0, in = 90, distance = 4.18879] (0.333333, 0)
	to[out = 90, in = 90, distance = 6.28318] (0.833333, 0)
	to[out = 90, in = 90, distance = 4.18879] (0.5, 0)
	to[out = -90, in = -90, distance = 4.18879] (0.166667, 0)
	to[out = -90, in = -90, distance = 6.28318] (0.666667, 0)
	to[out = -90, in = 180, distance = 4.18879] (0.91574, -0.277778);
\draw[help lines] (0.5, 0) circle (0.5);
\draw[fill] (0.0842603, 0.277778) circle (0.0166667);
\draw[fill] (0.91574, -0.277778) circle (0.0166667);
\draw[fill] (0, 0) circle (0.0166667);
\draw[fill] (1, 0) circle (0.0166667);
\end{tikzpicture}
    \caption{Example of decomposition.}
    \label{fig: example of decomposition 1}
\end{figure}

The language of 2-colored operads (see Section~\ref{sec: operad}) provides a convenient framework for describing the construction of meanders using rooted trees. We represent $M \circ_i M'$ by a rooted tree with two vertices labeled $M$ and $M'$, joined by a directed edge labeled $i$ that is oriented from $M$ to $M'$. For the operation $M \bullet_i M'$ we use a dashed edge labeled $i$ with the same orientation convention. An example of this construction, applied to the meander in Figure~\ref{fig: example of decomposition 1}(a), is presented in Figure~\ref{fig: example of constructing}.

\begin{figure}
\newcommand{\smartdraw}[6]{%
  \pgfmathsetmacro{\r}{1.9*veclen((#1)-(#3),(#2)-(#4))}%
  \pgfmathparse{#5?"dashed":""}\edef\T{\pgfmathresult}%
  \draw[ultra thick, ->, \T]
    ({#1-(#1-(#3))/\r},{#2-(#2-(#4))/\r}) --
    ({#3+(#1-(#3))/\r},{#4+(#2-(#4))/\r})
    node[pos=.5]{#6};
}
\newcommand{\ddx}{1.4}
\newcommand{\ddy}{1.5}

    \centering
    \begin{tikzpicture}[scale = 1.7]
    
    \smartdraw{0.5 + 2*\ddx}{0}{0.5 + 0*\ddx}{-\ddy}{1}{1\ \ \ \ \ \ \  }
    \smartdraw{0.5 + 2*\ddx}{0}{0.5 + 1*\ddx}{-\ddy}{1}{2\ \ \ \ \ }
    \smartdraw{0.5 + 2*\ddx}{0}{0.5 + 2*\ddx}{-\ddy}{0}{1\ \ \ \ }
    \smartdraw{0.5 + 2*\ddx}{0}{0.5 + 3*\ddx}{-\ddy}{1}{\ \ \ \ \ 3}
    \smartdraw{0.5 + 2*\ddx}{0}{0.5 + 4*\ddx}{-\ddy}{1}{\ \ \ \ \ \ \ \ 4}

    \smartdraw{0.5 + 0*\ddx}{-\ddy}{0.5 - 1*\ddx}{-2*\ddy}{1}{1\ \ \ \ \ }
    \smartdraw{0.5 + 0*\ddx}{-\ddy}{0.5 + 0*\ddx}{-2*\ddy}{1}{2\ \ \ \ }
    \smartdraw{0.5 + 0*\ddx}{-\ddy}{0.5 + 1*\ddx}{-2*\ddy}{1}{\ \ \ \ \ 3}
    \begin{scope}[shift={(2*\ddx, 0)}]
        \draw[thick] (0, 0) to (1, 0);
\draw[thick] (0.0842603, 0.277778)
	to[out = 0, in = 90, distance = 4.18879] (0.333333, 0)
	to[out = 90, in = 90, distance = 6.28318] (0.833333, 0)
	to[out = 90, in = 90, distance = 4.18879] (0.5, 0)
	to[out = -90, in = -90, distance = 4.18879] (0.166667, 0)
	to[out = -90, in = -90, distance = 6.28318] (0.666667, 0)
	to[out = -90, in = 180, distance = 4.18879] (0.91574, -0.277778);
\draw[help lines] (0.5, 0) circle (0.5);
\draw[fill] (0.0842603, 0.277778) circle (0.0166667);
\draw[fill] (0.91574, -0.277778) circle (0.0166667);
\draw[fill] (0, 0) circle (0.0166667);
\draw[fill] (1, 0) circle (0.0166667);
    \end{scope}
\begin{scope}[shift={(0, -\ddy)}]
\draw[thick] (0, 0) to (1, 0);
\draw[thick] (0.0520968, 0.222222)
	to[out = 0, in = 90, distance = 4.18879] (0.333333, 0)
	to[out = -90, in = -90, distance = 4.18879] (0.666667, 0)
	to[out = -90, in = -90, distance = 6.28318] (0.166667, 0)
	to[out = -90, in = -90, distance = 8.37758] (0.833333, 0)
	to[out = 90, in = 90, distance = 4.18879] (0.5, 0)
	to[out = 90, in = 180, distance = 6.28318] (0.947903, 0.222222);
\draw[help lines] (0.5, 0) circle (0.5);
\draw[fill] (0.0520968, 0.222222) circle (0.0166667);
\draw[fill] (0.947903, 0.222222) circle (0.0166667);
\draw[fill] (0, 0) circle (0.0166667);
\draw[fill] (1, 0) circle (0.0166667);
\end{scope}
\begin{scope}[shift={(\ddx, -\ddy)}]
\draw[thick] (0, 0) to (1, 0);
\draw[thick] (0.0520968, 0.222222)
	to[out = 0, in = 90, distance = 4.18879] (0.333333, 0)
	to[out = -90, in = -90, distance = 4.18879] (0.666667, 0)
	to[out = 90, in = 180, distance = 4.18879] (0.947903, 0.222222);
\draw[help lines] (0.5, 0) circle (0.5);
\draw[fill] (0.0520968, 0.222222) circle (0.0166667);
\draw[fill] (0.947903, 0.222222) circle (0.0166667);
\draw[fill] (0, 0) circle (0.0166667);
\draw[fill] (1, 0) circle (0.0166667);
\end{scope}
\begin{scope}[shift={(2*\ddx, -\ddy)}]
\draw[thick] (0, 0) to (1, 0);
\draw[thick] (0.0285955, 0.166667)
	to[out = 0, in = 90, distance = 10.472] (0.833333, 0)
	to[out = -90, in = -90, distance = 2.09439] (0.666667, 0)
	to[out = 90, in = 90, distance = 2.09439] (0.5, 0)
	to[out = -90, in = -90, distance = 2.09439] (0.333333, 0)
	to[out = 90, in = 90, distance = 2.09439] (0.166667, 0)
	to[out = -90, in = 180, distance = 10.472] (0.971405, -0.166667);
\draw[help lines] (0.5, 0) circle (0.5);
\draw[fill] (0.0285955, 0.166667) circle (0.0166667);
\draw[fill] (0.971405, -0.166667) circle (0.0166667);
\draw[fill] (0, 0) circle (0.0166667);
\draw[fill] (1, 0) circle (0.0166667);
\end{scope}
\begin{scope}[shift={(3*\ddx, -\ddy)}]
\draw[thick] (0, 0) to (1, 0);
\draw[thick] (0.0520968, 0.222222)
	to[out = 0, in = 90, distance = 4.18879] (0.333333, 0)
	to[out = -90, in = -90, distance = 4.18879] (0.666667, 0)
	to[out = 90, in = 180, distance = 4.18879] (0.947903, 0.222222);
\draw[help lines] (0.5, 0) circle (0.5);
\draw[fill] (0.0520968, 0.222222) circle (0.0166667);
\draw[fill] (0.947903, 0.222222) circle (0.0166667);
\draw[fill] (0, 0) circle (0.0166667);
\draw[fill] (1, 0) circle (0.0166667);
\end{scope}
\begin{scope}[shift={(4*\ddx, -\ddy)}]
\draw[thick] (0, 0) to (1, 0);
\draw[thick] (0.0520968, 0.222222)
	to[out = 0, in = 90, distance = 4.18879] (0.333333, 0)
	to[out = -90, in = -90, distance = 4.18879] (0.666667, 0)
	to[out = 90, in = 180, distance = 4.18879] (0.947903, 0.222222);
\draw[help lines] (0.5, 0) circle (0.5);
\draw[fill] (0.0520968, 0.222222) circle (0.0166667);
\draw[fill] (0.947903, 0.222222) circle (0.0166667);
\draw[fill] (0, 0) circle (0.0166667);
\draw[fill] (1, 0) circle (0.0166667);
\end{scope}
\begin{scope}[shift={(-\ddx, -2*\ddy)}]
\draw[thick] (0, 0) to (1, 0);
\draw[thick] (0.0417424, 0.2)
	to[out = 0, in = 90, distance = 2.51327] (0.2, 0)
	to[out = -90, in = -90, distance = 2.51327] (0.4, 0)
	to[out = 90, in = 90, distance = 2.51327] (0.6, 0)
	to[out = -90, in = -90, distance = 2.51327] (0.8, 0)
	to[out = 90, in = 180, distance = 2.51327] (0.958258, 0.2);
\draw[help lines] (0.5, 0) circle (0.5);
\draw[fill] (0.0417424, 0.2) circle (0.0166667);
\draw[fill] (0.958258, 0.2) circle (0.0166667);
\draw[fill] (0, 0) circle (0.0166667);
\draw[fill] (1, 0) circle (0.0166667);
\end{scope}
\begin{scope}[shift={(0*\ddx, -2*\ddy)}]
\draw[thick] (0, 0) to (1, 0);
\draw[thick] (0.0520968, 0.222222)
	to[out = 0, in = 90, distance = 4.18879] (0.333333, 0)
	to[out = -90, in = -90, distance = 4.18879] (0.666667, 0)
	to[out = 90, in = 180, distance = 4.18879] (0.947903, 0.222222);
\draw[help lines] (0.5, 0) circle (0.5);
\draw[fill] (0.0520968, 0.222222) circle (0.0166667);
\draw[fill] (0.947903, 0.222222) circle (0.0166667);
\draw[fill] (0, 0) circle (0.0166667);
\draw[fill] (1, 0) circle (0.0166667);
\end{scope}
\begin{scope}[shift={(1*\ddx, - 2*\ddy)}]
\draw[thick] (0, 0) to (1, 0);
\draw[thick] (0.0520968, 0.222222)
	to[out = 0, in = 90, distance = 4.18879] (0.333333, 0)
	to[out = -90, in = -90, distance = 4.18879] (0.666667, 0)
	to[out = 90, in = 180, distance = 4.18879] (0.947903, 0.222222);
\draw[help lines] (0.5, 0) circle (0.5);
\draw[fill] (0.0520968, 0.222222) circle (0.0166667);
\draw[fill] (0.947903, 0.222222) circle (0.0166667);
\draw[fill] (0, 0) circle (0.0166667);
\draw[fill] (1, 0) circle (0.0166667);
\end{scope}
\end{tikzpicture}
    \caption{Construction of a meander via a rooted tree.}
    \label{fig: example of constructing}
\end{figure}

A meander may admit many such tree presentations. We single out a \emph{canonical presentation} by imposing two constraints: (1) every vertex is labeled by a prime meander, and (2) there is no edge oriented from a snake to a direct snake (recall Definition~\ref{def: direct and inverse snakes}). 

\begin{thm}
     Each open meander can be canonically constructed using snakes and irreducible meanders. 
\end{thm}
\begin{proof}
    The existence of such a presentation is guaranteed by Lemma~\ref{lem: prime submeander exist}. 
    Uniqueness is proved by induction on the depth of the tree. If the depth is zero (i.\,e. the meander is prime), there is nothing to prove. Assume $M$ has a tree presentation of depth $d>0$ satisfying (1)--(2). Note that the leaves of the tree correspond to submeanders in $M$. Moreover, we argue that this set of leaves is uniquely determined and corresponds precisely to the set of all irreducible submeanders and maximal snakes in $M$.
    \begin{itemize}
        \item If an irreducible submeander of $M$ were not represented by a leaf, it would be part of an internal vertex. This would contradict constraint (1).
        \item If a leaf labeled with a snake represented a non-maximal snake, it would have to be a direct snake connected to another snake, as the only insertion that produces a snake is inserting a direct snake into a snake (this follows from Lemma~\ref{lem: snakes classification}). This would contradict constraint (2).
    \end{itemize}
    Therefore, the set of submeanders corresponding to the leaves is uniquely determined. If we cut these submeanders from $M$ we obtain a meander $M_1$ and its tree presentation of depth $d-1$ still satisfying (1)--(2). By the induction hypothesis, that presentation is unique, and therefore the original presentation of $M$ is unique as well.
\end{proof}

We will consider another way of constructing meanders. For this purpose, we need to introduce a new class of meanders. 
\begin{definition} \label{def: iterated snakes}
    A meander is called an \emph{iterated snake} if no irreducible meanders occur in its decomposition.
\end{definition}

\begin{figure}[h]
    \centering
    \begin{tikzpicture}
        \draw[thick] (0, 0) to (4, 0);
        \draw[ultra thick] (0.16697, 0.8) to[out = 0, in = 90, distance = 15.0796] (1.2, 0)
         to[out = -90, in = -90, distance = 5.02655] (0.8, 0)
         to[out = 90, in = 90, distance = 5.02655] (0.4, 0)
         to[out = -90, in = -90, distance = 25.1327] (2.4, 0)
         to[out = 90, in = 90, distance = 5.02655] (2, 0)
         to[out = -90, in = -90, distance = 5.02655] (1.6, 0)
         to[out = 90, in = 90, distance = 25.1327] (3.6, 0)
         to[out = -90, in = -90, distance = 5.02655] (3.2, 0)
         to[out = 90, in = 90, distance = 5.02655] (2.8, 0)
        to[out = -90, in = 180, distance = 15.0796] (3.76887, -0.933333, 0);
        \draw[help lines] (2, 0) circle (2);
        \draw[fill] (0.16697, 0.8) circle (0.057735);
        \draw[fill] (3.76887, -0.933333) circle (0.057735);
        \draw[fill] (0, 0) circle (0.057735);
        \draw[fill] (4, 0) circle (0.057735);
    \end{tikzpicture}
    \hspace{1.5cm}
    \begin{tikzpicture}
        \draw[thick] (0, 0) to (4, 0);
        \draw[ultra thick] (0.0833703, 0.571429) to[out = 0, in = 90, distance = 46.6751] (3.71429, 0)
         to[out = -90, in = -90, distance = 17.952] (2.28571, 0)
         to[out = 90, in = 90, distance = 3.59039] (2.57143, 0)
         to[out = -90, in = -90, distance = 3.59039] (2.85714, 0)
         to[out = 90, in = 90, distance = 3.59039] (3.14286, 0)
         to[out = -90, in = -90, distance = 3.59039] (3.42857, 0)
         to[out = 90, in = 90, distance = 17.952] (2, 0)
         to[out = -90, in = -90, distance = 17.952] (0.571429, 0)
         to[out = 90, in = 90, distance = 10.7712] (1.42857, 0)
         to[out = -90, in = -90, distance = 3.59039] (1.14286, 0)
         to[out = 90, in = 90, distance = 3.59039] (0.857143, 0)
         to[out = -90, in = -90, distance = 10.7712] (1.71429, 0)
         to[out = 90, in = 90, distance = 17.952] (0.285714, 0)
        to[out = -90, in = 180, distance = 46.6751] (3.88562, -0.666667, 0);
        \draw[help lines] (2, 0) circle (2);
        \draw[fill] (0.0833703, 0.571429) circle (0.057735);
        \draw[fill] (3.88562, -0.666667) circle (0.057735);
        \draw[fill] (0, 0) circle (0.057735);
        \draw[fill] (4, 0) circle (0.057735);
    \end{tikzpicture}
    \caption{Example of iterated snakes.}
    \label{fig: example of iterated snakes}
\end{figure}

Iterated snakes form a simple class of meanders (we discuss this class in detail in Section~\ref{sec: iterated snakes}). Examples of iterated snakes are shown in Figure~\ref{fig: example of iterated snakes}. Using iterated snakes, the uniqueness of the construction can be stated as follows: each meander admits a unique tree presentation in which the vertices are labeled by either irreducible meanders or iterated snakes, with the additional condition that no edge is oriented from an iterated snake to another iterated snake. 

\begin{thm}\label{thm: main}
    Each open meander can be canonically constructed using iterated snakes and irreducible meanders. 
\end{thm}

\begin{remark}
    In fact, the decomposition of non-singular meanders can be expressed without passing to the singular ones. To do this, we need to insert meanders with an even number of intersections into \emph{cups} --- submeanders of total order two. However, this approach is much less convenient: the description of the decomposition becomes very cumbersome, and many constructions become less natural. In particular, the equation for the generating function of the meander numbers is not so easy to write (see Theorem~\ref{thm: generating function equation}). 
\end{remark}

\subsection{Equation for the generating function.}
The construction of meanders via trees allows us to express the generating function of all meanders in terms of the generating functions of iterated snakes and of irreducible meanders.

We use the following notation. Let $\{A_{n,k}\}_{n\geq 0,\, k\geq 0}$ be a bivariate  sequence of numbers, and let $f(x,t) = \sum\limits_{n\geq 0, k\geq 0} A_{n,k}x^n t^k$ be its generating function. 
We decompose $f(x,t)$ into two parts (the ``odd'' part $f^{(1)}(x,t)$ and the ``even'' part $f^{(2)}(x,t)$) in the following way:
$$
f(x, t) = \underbrace{\sum\limits_{n\geq 0, k\geq 0} A_{2n+1,k}x^{2n+1} t^k}_{=:f^{(1)}(x,t)} + \underbrace{\sum\limits_{n\geq 0, k\geq 0} A_{2n,k}x^{2n} t^k}_{=:f^{(2)}(x,t)}.
$$
Let $f(x,t)$ and $g(x,t)$ be two bivariate generating functions with $g(0,0) = 0$. Then we use the following notation: 
$$
(f\boxdot g)(x,t):= f\left(g^{(1)}(x,t),\, g^{(2)}(x,t)\right).
$$

This notation has a direct combinatorial meaning for meanders.
The variable $x$ marks transverse intersections, and $t$ marks non-transverse ones. A meander with an odd number of transverse intersections can be inserted only at a transverse intersection, whereas one with an even number of transverse intersections can be inserted only at a non-transverse intersection. Thus insertion corresponds to the above composition of generating functions via $\boxdot$.


\begin{thm}\label{thm: generating function equation}
Let 
\begin{align*}
    &\gfM(x,t) = \sum_{n,k}\M_{n,k}x^nt^k, \\
    &\gfMIrr(x,t) = \sum_{n,k}\Mirr_{n,k}x^nt^k,\\
    &\gfMIS(x,t) = \sum_{n,k}\MIS_{n,k}x^nt^k
\end{align*}
be the generating functions for the numbers of equivalence classes of meanders, of irreducible meanders, and of iterated snakes, respectively. Then
\begin{align} \label{eq: generating function}
    \gfM(x,t) = x + t +\left(\gfMIrr\boxdot \gfM\right)(x,t) + \left(\gfMIS\boxdot\left(x+t+\left(\gfMIrr\boxdot \gfM\right)\right)\right)(x,t).
\end{align}
\end{thm}
\begin{proof}
By Theorem~\ref{thm: main}, every meander is obtained in exactly one of the following ways:
either it has total order $1$ (contributing the term $x+t$), or else it is represented by a tree whose root is labeled with
\begin{itemize}
  \item an irreducible meander, and its subtrees can represent arbitrary meanders (contributing $\gfMIrr \boxdot \gfM$);
  \item an iterated snake, and its subtrees are not rooted at iterated snakes (contributing $\gfMIS \boxdot (x+t+\gfMIrr \boxdot \gfM)$).
\end{itemize}
Insertions are encoded by $\boxdot$. Since the generating functions involved have zero constant term (there is no empty meander), these substitutions are well-defined. Collecting the three disjoint cases yields \eqref{eq: generating function}.
\end{proof}

\begin{remark}
    In equation~\eqref{eq: generating function} the generating function $\gfMIS(x,t)$ is known (see Section~\ref{sec: iterated snakes}), while for $\gfMIrr(x,t)$ only a few terms are known (see Section~\ref{sec: irreducible meanders}). 
\end{remark}

It follows from Theorem~\ref{thm: generating function equation} that knowing the numbers $\left\{\Mirr_{n,k}\right\}_{n,k}$ and $\left\{\MIS_{n,k}\right\}_{n,k}$ we could easily compute the numbers $\left\{\M_{n,k}\right\}_{n,k}$ themselves. As it is shown in Section~\ref{sec: iterated snakes}, $\left\{\MIS_{n,k}\right\}_{n,k}$ can be computed quite easily even for large values of $n$ and $k$. On the contrary, we do not know any efficient algorithm for finding the numbers $\left\{\Mirr_{n,k}\right\}_{n,k}$ other than brute force (with one exception, see Theorem~\ref{thm: cup 3 irreducible meanders}). Thus the problem of computing the numbers $\left\{\M_{n,k}\right\}_{n,k}$ is reduced to the problem of computing the numbers $\left\{\Mirr_{n,k}\right\}_{n,k}$.

\section{Iterated snakes}\label{sec: iterated snakes}
In this section, we discuss some properties of iterated snakes. In particular, we give an effective formula to calculate the number of equivalence classes of iterated snakes and discuss some results about the asymptotic behavior of these numbers. 

\begin{lemma} \label{lem: number of snakes}
    Let $\MS_{n,k}$ be the number of pairwise non-equivalent snakes of order~$(n,k)$. Then 
$$
\MS_{n,k} = \begin{cases}
0 & n+k<2, \\
\binom{n+k}{n} & n \text{ is even,} \\
2\binom{n+k}{n} & n \text{ is odd}.
\end{cases}
$$    
\end{lemma}
\begin{proof}
    Let $M$ be a snake of order~$(n,k)$. By Lemma~\ref{lem: snakes classification} the permutation of $M$ is either ${(1,2,\dots, n+k)}$ or ${(n+k, n+k-1,\dots, 1)}$. Arbitrary $n$ labels in the permutation can correspond to transverse intersections, and all the different ways of choosing such labels lead to different snakes. It remains to note that there are no inverse snakes of order $(2n, k)$ for $n\geq 0$.
\end{proof}
\begin{corollary}
Let $\gfMS(x,t) = \sum\limits_{n\geq 0, k \geq 0}\MS_{n,k}x^nt^k$ be the generating function of the numbers of equivalence classes of snakes. Then
$$
\gfMS(x,t) =-\frac{3}{2 (t+x-1)}-\frac{1}{2(x - t + 1)}-t-2 x-1.
$$
\end{corollary}

\begin{remark}\label{rmrk: are meander of order 1 snakes}
    If we consider meanders of total order one as snakes, the generating function and the corresponding numerical sequences take a somewhat more natural form. However, this requires additional caveats when discussing the factorization. Therefore, we prefer not to consider them as snakes.  
\end{remark}

Now, we can move on to the enumeration of iterated snakes.
\begin{thm}\label{thm: generating function of iterated snakes}
    Let $\MIS_{n,k}$ be the number of pairwise non-equivalent iterated snakes of order $(n,k)$, and let $\gfMIS(x,t)$ be the generating function of these numbers. Then 
    \begin{align}    \label{eq: generating function of odd iterated snakes}
    \gfMIS^{(1)}(x,t) = \gfMS^{(1)}\left(x +  \frac{\gfMIS^{(1)}(x,t)}{2},\ t\right), \\
     \label{eq: generating function of even iterated snakes}
    \gfMIS^{(2)}(x,t) = \gfMS^{(2)}\left(x +  \frac{\gfMIS^{(1)}(x,t)}{2},\ t\right).
    \end{align}
\end{thm}
\begin{proof}
    The uniqueness of the decomposition requires forbidding insertions of direct snakes into other snakes (see Section~\ref{subsec: decomposition}). Therefore, each iterated snake is obtained by a sequence of insertions of inverse snakes into some snake. Note that inverse snakes have an odd number of transverse intersections, so the generating function of inverse snakes is $\frac{1}{2}\gfMS^{(1)}(x,t)$. 
\end{proof}
\begin{remark}
    We can use the equations~\eqref{eq: generating function of odd iterated snakes}--\eqref{eq: generating function of even iterated snakes} to explicitly find $\gfMIS(x,t)$, but the resulting formula is too cumbersome, so we do not include it in this paper.
\end{remark}

We can use Theorem~\ref{thm: generating function of iterated snakes} to calculate the numbers $\left\{\MIS_{n,k}\right\}_{n,k}$ directly. To write the resulting formula, we need to introduce some notation. Let $n$ and $k$ be non-negative integers. Then
\begin{itemize}
    \item $\delta(n) := (n\ \mathrm{mod}\ 2)$;
    \item $\mu \vdash (n,k)$ means that $\mu = \left((a_1, b_1)^{s_{(a_1,b_1)}};(a_2, b_2)^{s_{(a_2,b_2)}};\dots;(a_r, b_r)^{s_{(a_r,b_r)}}\right)$ is a partition of $(n ,k)$, (i.\,e. ${\sum\limits_{(a,b)\in \mu} s_{(a, b)}a = n}$ and ${\sum\limits_{(a,b)\in \mu} s_{(a, b)}b = k}$);
    \item if $\mu \vdash (n,k)$, then $|\mu|:=\sum\limits_{(a,b)\in \mu} s_{(a, b)}$;
    \item if $\mu \vdash (n,k)$, then $\binom{n}{\mu} := \frac{n!}{(n - |\mu|)!\prod\limits_{(a,b)\in \mu} (s_{(a, b)}!)}$.
\end{itemize}

\begin{corollary} \label{cor: formula for iterated snakes}
    For $\left\{\MIS_{n,k}\right\}_{n,k}$ the following recurrence relation holds: 
    \begin{equation} \label{eq: recurrence formula for iterated snakes}
    \MIS_{n, k} = 
    \sum_{r = 1}^{\frac{n}{2}} 
    \sum_{l = 0}^{k}
    \MS_{2r+\delta(n), l}
     \left(
    \sum_{\mu}
    \binom{2r+\delta(n)}{\mu} 
    \prod_{(i_1, i_2) \in \mu}\left(\frac{\MIS_{2i_1+1, i_2}}{2}\right)^{s_{(i_1,i_2)}}
    \right),
\end{equation}
where the third summation goes through all partitions $\mu$ of  $\left(\frac{n-2r-\delta(n)}{2}, k - l\right)$, such that $|\mu| \leq 2r+\delta(n)$.   
\end{corollary}

We have used equations~\eqref{eq: generating function of odd iterated snakes} and~\eqref{eq: generating function of even iterated snakes} to calculate the numbers $\left\{\MIS_{n,k}\right\}_{n\leq 100,k \leq 30}$ and $\left\{\MIS_{n,0}\right\}_{n\leq 500}$. The calculation was done with a {C++} program (the code of the program and the results of the calculations are available at~\cite{Bcode}).

\subsection{Number sequences associated with iterated snakes}\label{sec: iterated snakes connections}
The numbers $\left\{\MIS_{n, 0}\right\}_{n\geq 2}$ are the same as the numbers of $\mathrm{P}$-graphs with $2n$ edges defined in the work~\cite{R86}\footnote{$\mathrm{P}$-graphs are nothing but a graph-theoretic reformulation of iterated snakes, so we omit the precise definition.} (see~\cite[A007165]{oeis}). It is a well-studied number sequence, so we know the exact asymptotics of its even and odd parts (see~\cite[A100327]{oeis} and~\cite[A003169]{oeis} and references therein):
\begin{align*}
    \MIS_{2n+1, 0} &\sim \frac{\sqrt{4046 + 1122\sqrt{17}}}{136\sqrt{\pi}}\left(\frac{71 + 17\sqrt{17}}{16}\right)^n n^{-\frac{3}{2}},\\
     \MIS_{2n, 0} &\sim \frac{\sqrt{33\sqrt{17} - 119}}{4\sqrt{34\pi}}\left(\frac{71 + 17\sqrt{17}}{16}\right)^n n^{-\frac{3}{2}}.
\end{align*}

\begin{remark}
    Note that the asymptotics of the numbers of non-singular iterated snakes with even and odd numbers of intersections are slightly different. We expect a similar phenomenon for irreducible meanders (this agrees with the results of numerical experiments: see Figure~\ref{fig: growth rate plot}).
\end{remark}

We found other interesting number sequences among the numbers of iterated snakes. The numbers $\left\{\MIS_{1, k}\right\}_{k\geq 1}$ coincide\footnote{To be fully consistent with this sequence, a meander of order $(1,0)$ must also be considered a snake (see Remark~\ref{rmrk: are meander of order 1 snakes}).} with the numbers~\cite[A007070]{oeis} which form the sequence of numbers satisfying the following recurrence formula: $a_n = 4a_{n-1} - 2a_{n-2}$ where $a_0 = 1$ and $a_1 = 4$. It would be interesting to find a combinatorial explanation for why this recurrence relation holds for $\left\{\MIS_{1, k}\right\}_{k\geq 1}$.

Another finding links the meander counting problem to a well-known open combinatorial problem --- the counting of polyominoes, defined in~\cite{G54}. In~\cite{CFMRR07} the authors introduce a subclass of polyominoes that can be enumerated via {2-compositions}. The sum of the entries in the top rows of all 2-compositions of $k$ (\cite[A181292]{oeis}) coincides with the numbers $\left\{\MIS_{2, k}\right\}_{k\geq 0}$. As in the previous case, at the moment we do not know why these sequences match.

\section{Irreducible meanders}\label{sec: irreducible meanders}
In this section, we prove some elementary properties of irreducible meanders, in particular about their asymptotics. 

As we said before, to calculate the numbers $\left\{\M_{n, k}\right\}_{n,k}$ one only needs to know $\left\{\Mirr_{n,k}\right\}_{n,k}$ (because the numbers $\left\{\MIS_{n,k}\right\}_{n,k}$ are easy to calculate). Unfortunately, we do not know any suitable way to calculate these numbers. We used a fairly straightforward brute-force algorithm (similar to the one described in~\cite{SL12}) to calculate $\left\{\M_{n,k}^{(Ir)}\right\}_{n,k}$ for $n + 2k \leq 38$. The results of the calculations can be found in~\cite{Bcode} and partially in Appendix~\ref{app: table}. In contrast to iterated snakes, no known numerical sequences were found among the numbers $\left\{\Mirr_{n,k}\right\}_{n,k}$.

For what follows, it will be useful for us to transform singular meanders into non-singular ones. Let us describe this procedure. There exists a projection 
$$c:\bigcup_{n, k \geq 0}\mathfrak{M}_{n,k} \to \bigcup_{n\geq 0}\mathfrak{M}_{n, 0}$$ defined by the insertion of a meander of order $(2,0)$ at each non-transverse intersection, so the meander of order $(n,k)$ is mapped to the non-singular meander of order $(n+2k, 0)$. Note that this projection is surjective, and if $M$ and $M'$ are non-equivalent irreducible meanders, then $c(M)$ is not equivalent to $c(M')$ (due to the uniqueness of the factorization). A non-singular meander of order~$(n, 0)$ is called \emph{almost irreducible} if it is an image of an irreducible meander under the map $c$ (examples of almost irreducible meanders are presented in Figure~\ref{fig: symmetry of irreducible meanders}). 

\begin{prop}
 $\Mirr_{n,k} \equiv 0\ \mathrm{mod}\ 2$.
\end{prop}
\begin{proof}
Let $M = (D, (p_1, p_2, p_3, p_4), (m,l))$ be an irreducible meander of order $(n,k)$. If $n$ is even, consider the meander $M' = (D, (p_3, p_4, p_1, p_2), (m,l))$. If $n$ is odd, let $M' = (D, (p_3, p_2, p_1, p_4), (m,l))$. In both cases $M'$ is irreducible. To conclude the proof we need to show that $M$ and $M'$ are not equivalent. Since there is a one-to-one correspondence between irreducible meanders of order $(n,k)$ and almost irreducible meanders of total order $N=n+2k$ with precisely $k$ cups, it suffices to show that $c(M)$ is not equivalent to $c(M')$. 

Let  $(a_1,a_2,\dots, a_{N})$ be the permutation of $c(M)$. If $n$ is odd, the permutation of $c(M')$ is $(a_N,a_{N-1},\dots, a_{1})$, so $M$ is not equivalent to $M'$. 

Now let us consider the case when $n$ is even. In this case the permutation of $c(M')$ is given by $(\sigma(a_{N}),\sigma(a_{N-1}),\dots,\sigma(a_1))$, where $\sigma(i) = N+1-i$; see the examples in Figure~\ref{fig: symmetry of irreducible meanders}. 

\begin{figure}[h]
       \centering
       \begin{tikzpicture}[scale = 3]
\draw[thick] (0, 0) to (1, 0);
\draw[ultra thick] (0.0355579, 0.185185)
	to[out = 0, in = 90, distance = 4.18879] (0.333333, 0)
	to[out = -90, in = -90, distance = 4.18879] (0.666667, 0)
	to[out = 90, in = 90, distance = 1.39626] (0.777778, 0)
	to[out = -90, in = -90, distance = 6.98132] (0.222222, 0)
	to[out = 90, in = 90, distance = 1.39626] (0.111111, 0)
	to[out = -90, in = -90, distance = 9.77384] (0.888889, 0)
	to[out = 90, in = 90, distance = 4.18879] (0.555556, 0)
	to[out = -90, in = -90, distance = 1.39626] (0.444444, 0)
	to[out = 90, in = 180, distance = 6.98132] (0.964442, 0.185185);
\draw[help lines] (0.5, 0) circle (0.5);
\draw[fill] (0.0355579, 0.185185) circle (0.0166667);
\draw[fill] (0.964442, 0.185185) circle (0.0166667);
\draw[fill] (0, 0) circle (0.0166667);
\draw[fill] (1, 0) circle (0.0166667);
\node (name) at (0.5, -0.6) {$(3,6,7,2,1,8,5,4)$};
\end{tikzpicture}
\hspace{2cm}
\begin{tikzpicture}[scale = 3]
\draw[thick] (0, 0) to (1, 0);
\draw[ultra thick] (0.0355579, 0.185185)
	to[out = 0, in = 90, distance = 6.98132] (0.555556, 0)
	to[out = -90, in = -90, distance = 1.39626] (0.444444, 0)
	to[out = 90, in = 90, distance = 4.18879] (0.111111, 0)
	to[out = -90, in = -90, distance = 9.77384] (0.888889, 0)
	to[out = 90, in = 90, distance = 1.39626] (0.777778, 0)
	to[out = -90, in = -90, distance = 6.98132] (0.222222, 0)
	to[out = 90, in = 90, distance = 1.39626] (0.333333, 0)
	to[out = -90, in = -90, distance = 4.18879] (0.666667, 0)
	to[out = 90, in = 180, distance = 4.18879] (0.964442, 0.185185);
\draw[help lines] (0.5, 0) circle (0.5);
\draw[fill] (0.0355579, 0.185185) circle (0.0166667);
\draw[fill] (0.964442, 0.185185) circle (0.0166667);
\draw[fill] (0, 0) circle (0.0166667);
\draw[fill] (1, 0) circle (0.0166667);
\node (name) at (0.5, -0.6) {$(5,4,1,8,7,2,3,6)$};
\end{tikzpicture}
       \caption{Examples of almost irreducible meanders $c(M)$ and $c(M')$.}
       \label{fig: symmetry of irreducible meanders}
\end{figure}

Let $\tilde{M} = (D, (p_1, p_2, p_3, p_4), (\tilde{m},l))$ be a non-singular meander equivalent to $c(M)$. Since for non-singular meanders the permutation uniquely determines the equivalence class of a meander, $\tilde{M}$ is equivalent to $c(M')$ if and only if, for each $i=1,\dots, N$, we have $a_i = N+1-a_{N+1-i}$. Assume that  $\tilde{M}$ is equivalent to $c(M')$. 
In that case $a_1 \leq \frac{N}{2}$. Otherwise, $a_N = N+1 - a_1 < \frac{N}{2}+1$, and the subarc of $\tilde{m}$ connecting $p_1$ to the point labeled $a_1$ would intersect the subarc of $\tilde{m}$ connecting $p_3$ to the point labeled $a_N$.

Let $r\geq1$ be the maximal integer such that $a_i \leq \frac{N}{2}$ for all $1 \leq i \leq r$ (thus $a_{N+1-i} > \frac{N}{2}$ for $1 \leq i \leq r$). If $a_{r+1} \neq a_{N-r}$, then the subarc of $\tilde{m}$ connecting the points labeled $a_r$ and $a_{r+1}$ would intersect the subarc of $\tilde{m}$ connecting the points labeled $a_{N-r+1}$ and $a_{N-r}$ (since both of these subarcs lie on the same side of $l$). Thus $a_{r+1} = a_{N-r}$, and therefore the permutation of $\tilde{M}$ is
$$
(a_1, \dots, a_{r-1}, a_r, a_{N-r}, a_{N-r+1},\dots, a_N),
$$
so $r=\frac{N}{2}$. By Lemma~\ref{lem: submeanders and permutations}, $\tilde{M}$ contains a submeander of total order $\frac{N}{2}$. Since an almost irreducible meander of total order $N$ has only submeanders of total orders $1$, $2$, and $N$, and since there are no almost irreducible meanders of total order less than $8$, the assumption that $c(M)$ is equivalent to $c(M')$ leads to a contradiction.
\end{proof}

\begin{prop}\label{prop: irreducible more than 2 cups}
 $\Mirr_{n,k} = 0$, if $k<3$. In particular, there are no non-singular irreducible meanders.
\end{prop}
\begin{proof}
    Closed meanders can be represented via a pair of words in Dyck language (see~\cite{LZ92}). A Dyck language consists of strings of balanced parentheses, where each opening parenthesis ``('' is correctly matched and nested with a closing parenthesis ``)''. This is a well-known object in combinatorics, so we limit ourselves to this brief description (for details, see any book on combinatorics, for example~\cite{S99}). To work with (open) meanders, we extend the alphabet by adding an extra symbol --- ``|''. The procedure of matching a non-singular meander~$M$ to a pair of words $(A_M, B_M)$ in an extended Dyck language is shown in Figure~\ref{fig: meanders and Dyck}.

\begin{figure}[h]
    \centering
\begin{tikzpicture}[scale = 4.5]
    \draw[thick] (0, 0) to (1, 0);
    \draw[ultra thick] (0.0196773, 0.138889) to[out = 0, in = 90, distance = 5.23599] (0.416667, 0)
        to[out = -90, in = -90, distance = 1.0472] (0.333333, 0)
        to[out = 90, in = 90, distance = 1.0472] (0.25, 0)
        to[out = -90, in = -90, distance = 3.14159] (0.5, 0)
        to[out = 90, in = 90, distance = 1.0472] (0.583333, 0)
        to[out = -90, in = -90, distance = 5.23599] (0.166667, 0)
        to[out = 90, in = 90, distance = 1.0472] (0.0833333, 0)
        to[out = -90, in = -90, distance = 7.33038] (0.666667, 0)
        to[out = 90, in = 90, distance = 3.14159] (0.916667, 0)
        to[out = -90, in = -90, distance = 1.0472] (0.833333, 0)
        to[out = 90, in = 90, distance = 1.0472] (0.75, 0)
        to[out = -90, in = 180, distance = 3.14159] (0.933013, -0.25);
    \draw[help lines] (0.5, 0) circle (0.5);
    \draw[fill] (0.0196773, 0.138889) circle (0.0166667);
    \draw[fill] (0.933013, -0.25) circle (0.0166667);
    \draw[fill] (0, 0) circle (0.0166667);
    \draw[fill] (1, 0) circle (0.0166667);
\end{tikzpicture}
\begin{tikzpicture}[scale = 1.5]
    \draw [->, ultra thick] (0, 0) to (0.7, 0);
    \draw [white, fill] (-0.5,-1.5) circle (0.01);
    \draw [white, fill] (1,1.5) circle (0.01);
\end{tikzpicture}
\begin{tikzpicture}[scale = 4.5]
    \newcommand{\deltayy}{0.2}
    \draw[white] (0.5, 0) circle (0.5);
    \node (A) at (0.5, 0.65*\deltayy) {$A_M = ()()|()(())$};
    \node (B) at (0.5, -0.7*\deltayy) {$B_M = (((())))|()$};
    \begin{scope}
        \clip (0,0+ \deltayy) rectangle (1, 0.2+\deltayy);
        \draw[thick] (0, 0 + \deltayy) to (1, 0 + \deltayy);
        \draw[ultra thick] (0.416667, 0.2 + \deltayy) to[out = -90, in = 90, distance = 5.23599] (0.416667, 0 + \deltayy)
        to[out = -90, in = -90, distance = 1.0472] (0.333333, 0 + \deltayy)
        to[out = 90, in = 90, distance = 1.0472] (0.25, 0 + \deltayy)
        to[out = -90, in = -90, distance = 3.14159] (0.5, 0 + \deltayy)
        to[out = 90, in = 90, distance = 1.0472] (0.583333, 0 + \deltayy)
        to[out = -90, in = -90, distance = 5.23599] (0.166667, 0 + \deltayy)
        to[out = 90, in = 90, distance = 1.0472] (0.0833333, 0 + \deltayy)
        to[out = -90, in = -90, distance = 7.33038] (0.666667, 0 + \deltayy)
        to[out = 90, in = 90, distance = 3.14159] (0.916667, 0 + \deltayy)
        to[out = -90, in = -90, distance = 1.0472] (0.833333, 0 + \deltayy)
        to[out = 90, in = 90, distance = 1.0472] (0.75, 0 + \deltayy)
        to[out = -90, in = 180, distance = 3.14159] (0.75, -0.2 + \deltayy);
    \end{scope}
    \begin{scope}
        \clip (0,0- \deltayy) rectangle (1, -0.2 - \deltayy);
        \draw[thick] (0, 0- \deltayy) to (1, 0- \deltayy);
        \draw[ultra thick] (0.416667, 0.2- \deltayy) to[out = -90, in = 90, distance = 5.23599] (0.416667, 0- \deltayy)
        to[out = -90, in = -90, distance = 1.0472] (0.333333, 0- \deltayy)
        to[out = 90, in = 90, distance = 1.0472] (0.25, 0- \deltayy)
        to[out = -90, in = -90, distance = 3.14159] (0.5, 0- \deltayy)
        to[out = 90, in = 90, distance = 1.0472] (0.583333, 0- \deltayy)
        to[out = -90, in = -90, distance = 5.23599] (0.166667, 0- \deltayy)
        to[out = 90, in = 90, distance = 1.0472] (0.0833333, 0- \deltayy)
        to[out = -90, in = -90, distance = 7.33038] (0.666667, 0- \deltayy)
        to[out = 90, in = 90, distance = 3.14159] (0.916667, 0- \deltayy)
        to[out = -90, in = -90, distance = 1.0472] (0.833333, 0- \deltayy)
        to[out = 90, in = 90, distance = 1.0472] (0.75, 0- \deltayy)
        to[out = -90, in = 90, distance = 3.14159] (0.75, -0.2- \deltayy);
    \end{scope}
\end{tikzpicture}

    \caption{Correspondence between non-singular meander and a pair of words in the extended Dyck language.}
    \label{fig: meanders and Dyck}
\end{figure}

    A \emph{cup} in a non-singular meander is a submeander of order $(2,0)$. In terms of Dyck language it corresponds to a substring of the form $()$. The presentation of a non-singular meander as a pair of words in Dyck language shows that each non-singular meander of order greater than two has at least two cups. Furthermore, a non-singular meander $M$ has precisely two cups only if the corresponding pair of words $(A_M, B_M)$ is either $A_M = |(\dots()\dots)$ and $B_M = (\dots()\dots)|$ or $A_M = (\dots()\dots)|$ and $B_M = |(\dots()\dots)$. In both cases, it contains a submeander of order $(3,0)$. It follows that each almost irreducible meander has at least three cups, and thus the number of non-transverse intersections in an irreducible meander is at least three.
\end{proof}

\begin{thm}\label{thm: cup 3 irreducible meanders}
$$\Mirr_{n,3} =
\begin{cases}
    0 &  n \text{ is odd,}\\
    \varphi\left(\frac{n}{2}+4\right) - 2 & n \text{ is even,}
\end{cases}$$ where $\varphi(x)$ is Euler's totient function.
\end{thm}
\begin{proof}
Let $M$ be an almost irreducible meander of total order $n+6$ with precisely three cups (each such meander corresponds to an irreducible meander of order $(n, 3)$). We can encode $M$ with a pair of words $(A_M, B_M)$ in extended Dyck language, as it was done in the proof of Proposition~\ref{prop: irreducible more than 2 cups}. Since $M$ is almost irreducible, neither $A_M$ nor $B_M$ begins or ends with ``$|$''; otherwise the associated permutation would begin or end with $1$ or $n+6$, and hence, by Lemma~\ref{lem: submeanders and permutations}, $M$ would contain a submeander of total order $n+5$. It follows that 
\begin{align*}
    A_M &= 
    \underbrace{(\,(\,\dots\,(}_{n_1}\,
    \underbrace{)\,)\,\dots\,)}_{n_1}\,
    |\,|\,
    \underbrace{(\,(\,\dots\,(}_{n_2}\,
    \underbrace{)\,)\,\dots\,)}_{n_2},\\
    B_M &= 
    \underbrace{(\,(\,\dots\,(}_{n_1+n_2+1}\,
    \underbrace{)\,)\,\dots\,)}_{n_1+n_2+1},
\end{align*}
where $n + 6 = 2(n_1+n_2+1)$ with $n_1$ and $n_2$ being positive. In particular, we see that $n$ must be even. Thus, we can encode almost irreducible meanders with three cups via a pair of positive integers $(n_1, n_2)$. We show that such a pair of numbers corresponds to a meander if and only if $n_1+1$ and $n_2+1$ are coprime. The statement of the theorem follows from this immediately. Indeed, 
$$
\operatorname{GCD}(n_1+1, n_2+1) = \operatorname{GCD}\left(n_1+1, \frac{n}{2}+2 - n_1 +1\right) = \operatorname{GCD}\left(n_1+1, \frac{n}{2}+4\right).
$$
Thus we obtain
\begin{align*}
\Mirr_{n,3}
=& \sum_{\substack{n_1, n_2 > 0\\n + 6 = 2(n_1+n_2+1)\\ \operatorname{GCD}(n_1+1, n_2+1) = 1}} 1 
=\sum_{\substack{n_1 = 1\\ \operatorname{GCD}\left(n_1+1, \frac{n}{2}+4\right) = 1}}^{\frac{n}{2} + 1} 1 
=\sum_{\substack{n_1 = 2\\ \operatorname{GCD}\left(n_1, \frac{n}{2}+4\right) = 1}}^{\frac{n}{2} + 2} 1 \\
=& \varphi\left(\frac{n}{2}+4\right) - 2.
\end{align*}
To demonstrate that a pair $(n_1,n_2)$ corresponds to a meander if and only if $n_1+1$ and $n_2+1$ are coprime, we consider a meander with three cups as a curve in a disk with punctures. Such curves arose in the study of braid groups (see for example~\cite {D02}, \cite{M04}, \cite{DW07}), and in particular such curves have been used to develop efficient algorithms for comparing braids. We use a simple special case of such an algorithm. For completeness, we mention that these are particular cases of a more general technique introduced by Agol, Hass, and Thurston~\cite{AHT06}.

To a pair of numbers $(n_1, n_2)$ we associate a pair of words in the extended Dyck language as above. Using such a pair we can construct a set of curves in a disk (as was done for meanders). Let us connect the two free points in order to obtain a set of closed curves (the number of connected components would not change), which we denote by $X_{(n_1, n_2)}$ (see the example in Figure~\ref{fig: simplify X}). Without loss of generality, we can assume that $n_1 \geq n_2$. We consider several cases.

\begin{enumerate}
    \item If $n_1 = n_2 \neq 0$, then $X_{(n_1, n_2)}$ has several connected components, and therefore $(n_1, n_2)$ does not correspond to a meander.
    \item If $n_2 = 0$, then $X_{(n_1, n_2)}$ is a single curve. (This case does not correspond to an almost irreducible meander with three cups, but we will need this case for later analysis).
    \item In general case we can simplify $X_{(n_1, n_2)}$ by pulling $2n_2 + 2$ arcs resulting in $X_{(n_2, n_1-n_2-1)}$, see Figure~\ref{fig: simplify X}. Consequently, $X_{(n_1, n_2)}$ is a single curve if and only if $X_{(n_2, n_1-n_2-1)}$ is a single curve. That is, we almost get Euclid's algorithm: consider a sequence $\{n_i\}_{i=1,2,\dots}$ where for each $i>2$ $n_i = n_{i-2}\ \operatorname{mod}\ (n_{i-1} + 1)$. If $n_i = n_{i+1} \neq 0$ for some $i$, then $X_{(n_1, n_2)}$ has several connected components (according to the first case). Otherwise, $n_i = 0$ for some $i$ and $X_{(n_1, n_2)}$ is a single curve (according to the second case). The only thing left to note is that $\{n_i\}_{i=1,2,\dots}$ stabilizes with zeroes if and only if $n_1 + 1$ and $n_2 +1$ are coprime.
\end{enumerate}
\begin{figure}
    \centering
    \begin{tikzpicture}[scale = 0.2]
    \draw[thick] (-0.5, 0) to (23.5, 0);
    \begin{scope}
        \clip (-0.5,0) rectangle (23.5,11.6);
        \draw[thick] (7.5,0) circle (0.5);
        \draw[thick] (7.5,0) circle (1.5);
        \draw[thick] (7.5,0) circle (2.5);
        \draw[thick] (7.5,0) circle (3.5);
        \draw[thick] (7.5,0) circle (4.5);
        \draw[thick] (7.5,0) circle (5.5);
        \draw[thick] (7.5,0) circle (6.5);
        \draw[thick] (7.5,0) circle (7.5);

        \draw[thick] (16.5,0) circle (0.5);

        \draw[thick] (20.5,0) circle (0.5);
        \draw[thick] (20.5,0) circle (1.5);
        \draw[thick] (20.5,0) circle (2.5);
    \end{scope}
    \begin{scope}
        \clip (-0.5,0) rectangle (23.5,-11.6);
        \draw[thick] (11.5,0) circle (0.5);
        \draw[thick] (11.5,0) circle (1.5);
        \draw[thick] (11.5,0) circle (2.5);
        \draw[thick] (11.5,0) circle (3.5);
        \draw[thick] (11.5,0) circle (4.5);
        \draw[thick] (11.5,0) circle (5.5);
        \draw[thick] (11.5,0) circle (6.5);
        \draw[thick] (11.5,0) circle (7.5);
        \draw[thick] (11.5,0) circle (8.5);
        \draw[thick] (11.5,0) circle (9.5);
        \draw[thick] (11.5,0) circle (10.5);
        \draw[thick] (11.5,0) circle (11.5);
    \end{scope}
    \node (name) at (11.5, -14) {$X_{(8,3)}$};
\end{tikzpicture}
\begin{tikzpicture}[scale = 0.2]
    \begin{scope}
        \clip (-1.5,-15) rectangle (5,11.6);
        \draw[ultra thick, ->] (-0.5,0) to (4, 0);
    \end{scope}  
\end{tikzpicture}     
\begin{tikzpicture}[scale = 0.2]
    \draw[thick] (-0.5, 0) to (23.5, 0);
    \begin{scope}
        \clip (-0.5,0) rectangle (23.5,11.6);
        \draw[thick] (7.5,0) circle (0.5);
        \draw[thick] (7.5,0) circle (1.5);
        \draw[thick] (7.5,0) circle (2.5);
        \draw[thick] (7.5,0) circle (3.5);
        \draw[thick] (7.5,0) circle (4.5);
        \draw[thick] (7.5,0) circle (5.5);
        \draw[thick] (7.5,0) circle (6.5);
        \draw[thick] (7.5,0) circle (7.5);
    \end{scope}   
    \begin{scope}
        \clip (-0.5,0) rectangle (11.5,-11.6);
        \draw[thick] (11.5,0) circle (0.5);
        \draw[thick] (11.5,0) circle (1.5);
        \draw[thick] (11.5,0) circle (2.5);
        \draw[thick] (11.5,0) circle (3.5);
        \draw[thick] (11.5,0) circle (4.5);
        \draw[thick] (11.5,0) circle (5.5);
        \draw[thick] (11.5,0) circle (6.5);
        \draw[thick] (11.5,0) circle (7.5);
        \draw[thick] (11.5,0) circle (8.5);
        \draw[thick] (11.5,0) circle (9.5);
        \draw[thick] (11.5,0) circle (10.5);
        \draw[thick] (11.5,0) circle (11.5);
    \end{scope}
    \begin{scope}
        \clip (11.5,0) rectangle (23.5,-15.5);
        \draw[thick] (11.5,0) circle (0.5);
        \draw[thick] (11.5,0) circle (1.5);
        \draw[thick] (11.5,0) circle (2.5);        
        \draw[thick] (11.5,0) circle (3.5);

        \draw[thick] (11.5,-5) circle (0.5);

        \draw[thick] (11.5,-9) circle (0.5);
        \draw[thick] (11.5,-9) circle (1.5);
        \draw[thick] (11.5,-9) circle (2.5);
    \end{scope}
    \draw[ultra thick, dashed,  ->] (20.5, 2) to[out = -90, in = 0, looseness = 1.2] (15, -9);
    \draw[ultra thick, dashed,  ->] (16.5, 1) to[out = -90, in = 0, looseness = 1.2] (13, -5);
\end{tikzpicture}

\vspace{-1cm}
\begin{tikzpicture}[scale = 0.2]
    \begin{scope}
        \clip (-7,-2) rectangle (7,3);
        \draw[ultra thick, ->] (7,3) to (4, -1);
    \end{scope}  
\end{tikzpicture}   

\vspace{-0.3cm}
\begin{tikzpicture}[scale = 0.2]
    \draw[thick] (-0.5, 0) to (15.5, 0);
    \begin{scope}
        \clip (-0.5,0) rectangle (15.5,8.6);
        \draw[thick] (7.5,0) circle (0.5);
        \draw[thick] (7.5,0) circle (1.5);
        \draw[thick] (7.5,0) circle (2.5);
        \draw[thick] (7.5,0) circle (3.5);
        \draw[thick] (7.5,0) circle (4.5);
        \draw[thick] (7.5,0) circle (5.5);
        \draw[thick] (7.5,0) circle (6.5);
        \draw[thick] (7.5,0) circle (7.5);

    \end{scope}   
    \begin{scope}
        \clip (-0.5,0) rectangle (15.5,-5.6);
        \draw[thick] (2.5,0) circle (0.5);
        \draw[thick] (2.5,0) circle (1.5);
        \draw[thick] (2.5,0) circle (2.5);
        
        \draw[thick] (6.5,0) circle (0.5);
        
        \draw[thick] (11.5,0) circle (0.5);
        \draw[thick] (11.5,0) circle (1.5);
        \draw[thick] (11.5,0) circle (2.5);
        \draw[thick] (11.5,0) circle (3.5);
    \end{scope}
    \node (name) at (7.5, -5) {$X_{(4,3)}$};
\end{tikzpicture}
    \caption{An example of simplification of $X_{(n_1, n_2)}$.}
    \label{fig: simplify X}
\end{figure}
\end{proof}

\subsection{Asymptotics of the numbers of irreducible meanders} 
In this subsection, we show that the number of pairwise non-equivalent irreducible meanders with a fixed number of non-transverse intersections grows at most polynomially. By contrast, the growth rate of almost irreducible meanders is exponential.

 \begin{thm}\label{thm: irreducible meander growth rate}
    For a fixed natural number $k$ 
    $$\sum_{n=1}^N\Mirr_{n,k} = O(N^{2k-4}).$$
\end{thm}
To prove this theorem we need the following lemma.
\begin{lemma}\label{lem: meander numbers with cups}
 Let $k$ be a natural number greater than two, and let $\mathcal{B}_{n, k}$ be the number of pairwise non-equivalent non-singular meanders with precisely $n$ intersections and $k$ cups. Then 
    \begin{align*}
        &\mathcal{B}_{2n+1, k} \leq \mathcal{B}_{2n+2, k},\\
        &\mathcal{B}_{2n, k} \leq 2\mathcal{B}_{2n+1, k}.
    \end{align*}        
    \end{lemma}
\begin{proof}
    Let $M$ be a non-singular meander of total order $2n+1$ with precisely $k$ cups, and let $(\alpha_1, \dots, \alpha_{2n+1})$ be the permutation of $M$. If $\alpha_{2n+1} \neq 2n+1$ then a non-singular meander $M'$ with permutation $(\alpha_1, \dots, \alpha_{2n+1}, 2n+2)$ has the same number of cups. If $\alpha_{2n+1} = 2n+1$, then a non-singular meander $M'$ with permutation $(1, \alpha_{2n+1} + 1, \alpha_{2n} + 1, \dots, \alpha_{1} + 1)$ has the same number of cups. Clearly, for different $M$ we obtain different $M'$, and thus $\mathcal{B}_{2n+1, k} \leq \mathcal{B}_{2n+2, k}$.

    Let $M$ be a non-singular meander of total order $2n$ with precisely $k$ cups, and let $(\alpha_1, \dots, \alpha_{2n})$ be the permutation of $M$. We need to consider several cases.
    \begin{enumerate}
        \item Case $\alpha_{2n} \neq 2n$. Consider a non-singular meander $M'$ with permutation $(\alpha_1, \dots, \alpha_{2n}, 2n+1)$.
        \item Case $\alpha_{2n} = 2n$, but $\alpha_{1} \neq 1$. Consider a non-singular meander $M'$  with permutation $(\alpha_{2n}+1, \alpha_{2n-1}+1, \dots, \alpha_{1}+1, 1)$.
        \item Case $\alpha_{2n} = 2n$, $\alpha_{1} = 1$, but $\alpha_{2n-1}\neq 2n-1$. Consider a non-singular meander $M'$ with permutation $(1, \alpha_{2n-1}+1, \alpha_{2n-2}+1, \dots, \alpha_{1}+1, \alpha_{2n}+1)$.
        \item Case $\alpha_{2n} = 2n$, $\alpha_{1} = 1$, and $\alpha_{2n-1} = 2n-1$. Consider a non-singular meander $M'$ with permutation $(\alpha_{1}+2, \alpha_{2}+2, \dots, \alpha_{2n-1}+2, 2, 1)$.
    \end{enumerate}
    In all cases $M'$ is a non-singular meander of total order $2n+1$ and with precisely $k$ cups.
    The only possibilities where different $M$ can lead to equivalent $M'$ are cases 2 and 4. This means that each meander of total order $2n+1$ with precisely $k$ cups corresponds to at most two different meanders of total order $2n$ with precisely $k$ cups. It follows that $\mathcal{B}_{2n, k} \leq 2\mathcal{B}_{2n+1, k}$.
\end{proof}
\begin{proof}[Proof of Theorem~\ref{thm: irreducible meander growth rate}]    
    In~\cite{DGZZ20} it was proved that $\sum_{n=1}^{N}\mathcal{B}_{2n+1, k} = O(N^{2k-4})$.
    From Lemma~\ref{lem: meander numbers with cups} it follows that $\sum_{n=1}^N\mathcal{B}_{n, k} = O(N^{2k-4})$. It remains to note that $\Mirr_{n, k}\leq \mathcal{B}_{n+2k, k} $. 
\end{proof}
\begin{corollary} For each $k\geq 0$ 
    $$
    \lim_{n\to \infty} \frac{\Mirr_{n, k}}{\M_{n,k}} = 0.
    $$
\end{corollary}
\begin{proof}
     It is clear that $\M_{n,k} \geq \M_{n,0}$: indeed, starting from a non-singular meander with the permutation $(\alpha_1, \dots, \alpha_{n})$, we can construct a meander with the permutation  $(1, 2, \dots, k, k+\alpha_1, \dots, k+\alpha_{n})$, where intersections with the labels $1, 2,\dots, k$ are non-transverse. But $\M_{n,0}$ grows exponentially with $n$ (see~\cite{AP05}).
\end{proof}

We denote the number of all pairwise non-equivalent almost irreducible meanders of total order $n$ by $\A_n$. Note that $\A_n = \sum\limits_{r+2k=n}\Mirr_{r,k}$. In Figure~\ref{fig: growth rate plot} one can see the logarithm of the proportion of almost irreducible meanders among all non-singular meanders of total order less than $39$.

\begin{figure}
    \centering
    \includegraphics[width=\textwidth]{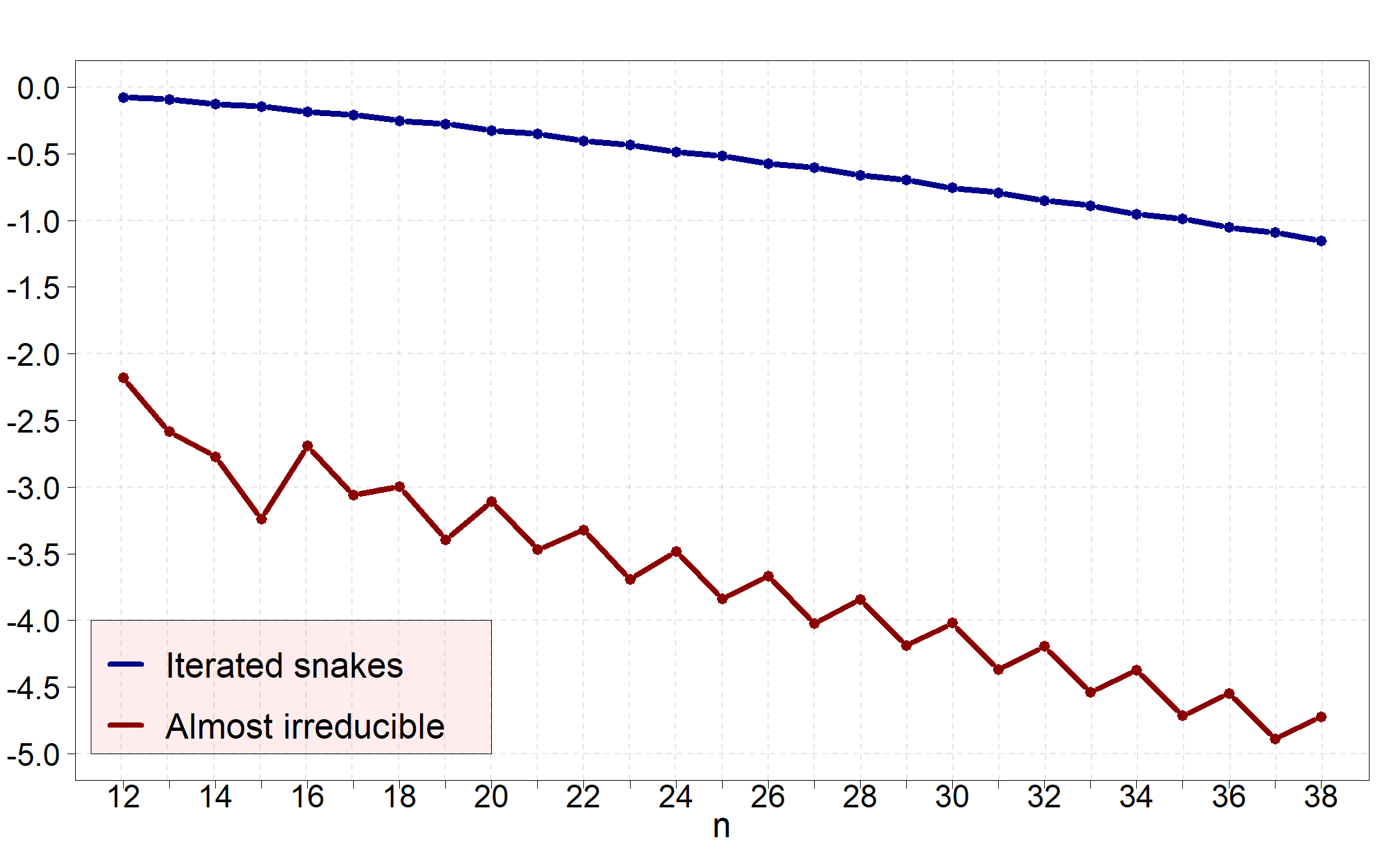}
    \caption{Base 10 logarithm of proportions of almost irreducible meanders and of non-singular iterated snakes among all non-singular meanders.}
    \label{fig: growth rate plot}
\end{figure}

The following estimates of the growth rate of $\A_n$ were obtained jointly with A.~Malyutin. The author is grateful for permission to include them in the present paper. 
\begin{thm}\label{thm: upper bound}
\begin{equation}
    \limsup_{n\to \infty} \sqrt[n]{\A_n} < 3.313385
\end{equation}
\end{thm}
\begin{proof}
    Let $M$ be an almost irreducible meander of total order $n$. Let us choose $\alpha \in (0;\,1)$ such that $\alpha n$ is a natural number, and let $\mu = (a_1^{s_1}; a_2^{s_2};\dots; a_k^{s_k})$ be the partition of $\alpha n$. We can choose $\sum_{i=1}^k s_i$ distinct intersections in $M$ and at each of them insert a non-singular inverse snake of total order $2a_i+1$. As a result, we obtain a non-singular meander of total order $n+\sum_{i=1}^r 2a_i = n+2\alpha n$. In this way, we obtain $\sum\limits_{\mu \vdash \alpha n} \binom{n}{\mu} = \binom{n+\alpha n-1}{\alpha n}$ non-equivalent non-singular meanders of total order $n+2\alpha n$ from a single almost irreducible meander of total order~$n$. According to the uniqueness of the decomposition, different almost irreducible meanders also lead to different meanders. We have the following inequalities:
\begin{align*}
    \binom{n+\alpha n-1}{\alpha n}\A_n &\leq \M_{n+2\alpha n, 0},\\
    \limsup_{n\to \infty}  \sqrt[n]{\binom{n+\alpha n-1}{\alpha n}\A_n} & \leq \lim_{n\to \infty} \sqrt[n]{\M_{n+2\alpha n, 0}} \\
    \frac{(1+\alpha)^{1+\alpha}}{\alpha^\alpha} \limsup_{n\to \infty} \sqrt[n]{\A_n}  &\leq
    \left(\lim_{n\to \infty} \sqrt[n]{\M_{n,0}}\right)^{1 + 2\alpha},   \\
    \limsup_{n\to \infty} \sqrt[n]{\A_n} &\leq \frac{\alpha^\alpha}{(1+\alpha)^{1+\alpha}} \left(\lim_{n\to \infty} \sqrt[n]{\M_{n,0}}\right)^{1+2\alpha}.
\end{align*}
    It is proved in \cite{AP05} that $\lim\limits_{n\to \infty} \sqrt[n]{\Bar{\M}_{n}}\leq 12.901$, where $\Bar{\M}_{n}$ is the number of pairwise non-equivalent closed meanders with precisely $2n$ intersections. It can be easily seen that $\M_{2n-1, 0} = \Bar{\M}_{n}$ and $\Bar{\M}_{n} \leq \M_{2n, 0} \leq n\Bar{\M}_{n}$ (see, for example,~\cite{C03} for details). From this it follows: $\lim\limits_{n\to \infty} \sqrt[n]{\M_{n, 0}} \leq \sqrt{12.901}$. 
    Now for each $\alpha \in (0;\,1)$ we have 
\begin{align}
\label{eq: lower bound}
     \limsup_{n\to \infty} \sqrt[n]{\A_n} &\leq \frac{\alpha^\alpha}{(1+\alpha)^{1+\alpha}} \left(12.901\right)^{\frac{1+2\alpha}{2}}.
\end{align}
    The function on the right side of equation~$\eqref{eq: lower bound}$ reaches the minimum at $\alpha \approx \frac{10}{119}$, where its value is approximately 3.313384. 
\end{proof}

\begin{corollary} 
    $$
    \lim_{n\to \infty} \frac{\A_n}{\M_{n,0}} = 0.
    $$
\end{corollary}
\begin{proof}
    From~\cite{AP05} it follows that $\lim\limits_{n\to \infty} \sqrt[n]{\M_{n, 0}} > 3.37$. It remains to be noted that $\limsup\limits_{n\to \infty} \sqrt[n]{\A_n}~<~\lim\limits_{n\to \infty} \sqrt[n]{\M_{n, 0}}.$
\end{proof}

\begin{thm}
\begin{equation}
    \liminf_{n\to \infty} \sqrt[n]{\A_n} > 1.83669.
\end{equation}
\end{thm}
\begin{proof}
    First, let us non-formally describe the main idea of the proof.
    Let $M$ be an arbitrary non-singular meander of total order $n$, where $n$ is odd (for even $n$ the idea is the same). We can construct almost irreducible meanders of orders $2n+20$ and $2n+23$ from $M$ by the following procedure. 
    Consider a meander $M_1$ that is a concatenation of a non-singular meander with the permutation $(7, 6, 1, 2, 5, 8, 9, 4, 3)$ and $M$ (see the example in Figure~\ref{fig: construct almost irreducible odd}(a), where $M$ is the meander with the permutation $(1,2,3,4,5)$). Next, we need to ``double'' $M_1$ (as in Figure~\ref{fig: construct almost irreducible odd}(b)) to obtain a meander $M_2$. Finally, we can transform $M_2$ into an almost irreducible meander $M_3$ of total order $2(n+9)+2$ by adding two more intersections between points with labels $14$ and $15$ (see Figure~\ref{fig: construct almost irreducible odd}(c)). We can also obtain an almost irreducible meander $M_4$ of odd total order in a similar way, see Figure~\ref{fig: construct almost irreducible odd}(d).

    \begin{figure}[h]
    \centering
    \begin{tikzpicture}[scale = 4.3]
    \node (a) at (0.5, -0.6) {(a)};
\draw[thick] (0, 0) to (1, 0);
\draw[ultra thick] (0.0248132, 0.155556) to[out = 0, in = 90, distance = 5.86431] (0.466667, 0)
 to[out = -90, in = -90, distance = 0.837758] (0.4, 0)
 to[out = 90, in = 90, distance = 4.18879] (0.0666667, 0)
 to[out = -90, in = -90, distance = 0.837758] (0.133333, 0)
 to[out = 90, in = 90, distance = 2.51327] (0.333333, 0)
 to[out = -90, in = -90, distance = 2.51327] (0.533333, 0)
 to[out = 90, in = 90, distance = 0.837758] (0.6, 0)
 to[out = -90, in = -90, distance = 4.18879] (0.266667, 0)
 to[out = 90, in = 90, distance = 0.837758] (0.2, 0)
 to[out = -90, in = -90, distance = 5.86431] (0.666667, 0)
 to[out = 90, in = 90, distance = 0.837758] (0.733333, 0)
 to[out = -90, in = -90, distance = 0.837758] (0.8, 0)
 to[out = 90, in = 90, distance = 0.837758] (0.866667, 0)
 to[out = -90, in = -90, distance = 0.837758] (0.933333, 0)
to[out = 90, in = 180, distance = 0.837758] (0.975187, 0.155556);
\draw[help lines] (0.5, 0) circle (0.5);
\draw[fill] (0.0248132, 0.155556) circle (0.0129099);
\draw[fill] (0.975187, 0.155556) circle (0.0129099);
\draw[fill] (0, 0) circle (0.0129099);
\draw[fill] (1, 0) circle (0.0129099);
\end{tikzpicture}
\hspace{0.7cm}
\begin{tikzpicture}[scale = 4.3]
    \node (a) at (0.5, -0.6) {(b)};
\draw[thick] (0, 0) to (1, 0);
\draw[ultra thick] (0.02285, 0.149425) to[out = 0, in = 90, distance = 5.6332] (0.448276, 0)
 to[out = -90, in = -90, distance = 0.433323] (0.413793, 0)
 to[out = 90, in = 90, distance = 4.76655] (0.0344828, 0)
 to[out = -90, in = -90, distance = 1.29997] (0.137931, 0)
 to[out = 90, in = 90, distance = 2.16662] (0.310345, 0)
 to[out = -90, in = -90, distance = 3.03326] (0.551724, 0)
 to[out = 90, in = 90, distance = 0.433323] (0.586207, 0)
 to[out = -90, in = -90, distance = 3.89991] (0.275862, 0)
 to[out = 90, in = 90, distance = 1.29997] (0.172414, 0)
 to[out = -90, in = -90, distance = 6.49985] (0.689655, 0)
 to[out = 90, in = 90, distance = 0.433323] (0.724138, 0)
 to[out = -90, in = -90, distance = 1.29997] (0.827586, 0)
 to[out = 90, in = 90, distance = 0.433323] (0.862069, 0)
 to[out = -90, in = -90, distance = 1.29997] (0.965517, 0)
 to[out = 90, in = 90, distance = 0.433323] (0.931034, 0)
 to[out = -90, in = -90, distance = 0.433323] (0.896552, 0)
 to[out = 90, in = 90, distance = 1.29997] (0.793103, 0)
 to[out = -90, in = -90, distance = 0.433323] (0.758621, 0)
 to[out = 90, in = 90, distance = 1.29997] (0.655172, 0)
 to[out = -90, in = -90, distance = 5.6332] (0.206897, 0)
 to[out = 90, in = 90, distance = 0.433323] (0.241379, 0)
 to[out = -90, in = -90, distance = 4.76655] (0.62069, 0)
 to[out = 90, in = 90, distance = 1.29997] (0.517241, 0)
 to[out = -90, in = -90, distance = 2.16662] (0.344828, 0)
 to[out = 90, in = 90, distance = 3.03326] (0.103448, 0)
 to[out = -90, in = -90, distance = 0.433323] (0.0689655, 0)
 to[out = 90, in = 90, distance = 3.89991] (0.37931, 0)
 to[out = -90, in = -90, distance = 1.29997] (0.482759, 0)
to[out = 90, in = 180, distance = 6.49985] (0.97715, 0.149425);
\draw[help lines] (0.5, 0) circle (0.5);
\draw[fill] (0.02285, 0.149425) circle (0.0129099);
\draw[fill] (0.97715, 0.149425) circle (0.0129099);
\draw[fill] (0, 0) circle (0.0129099);
\draw[fill] (1, 0) circle (0.0129099);
\end{tikzpicture}
\vspace{0.3cm}

\begin{tikzpicture}[scale = 4.3]
    \node (a) at (0.5, -0.6) {(c)};
\draw[thick] (0, 0) to (1, 0);
\draw[ultra thick] (0.0346123, 0.182796) to[out = 0, in = 90, distance = 5.26977] (0.419355, 0)
 to[out = -90, in = -90, distance = 0.405367] (0.387097, 0)
 to[out = 90, in = 90, distance = 4.45903] (0.0322581, 0)
 to[out = -90, in = -90, distance = 1.2161] (0.129032, 0)
 to[out = 90, in = 90, distance = 2.02683] (0.290323, 0)
 to[out = -90, in = -90, distance = 3.6483] (0.580645, 0)
 to[out = 90, in = 90, distance = 0.405367] (0.612903, 0)
 to[out = -90, in = -90, distance = 4.45903] (0.258065, 0)
 to[out = 90, in = 90, distance = 1.2161] (0.16129, 0)
 to[out = -90, in = -90, distance = 6.89124] (0.709677, 0)
 to[out = 90, in = 90, distance = 0.405367] (0.741935, 0)
 to[out = -90, in = -90, distance = 1.2161] (0.83871, 0)
 to[out = 90, in = 90, distance = 0.405367] (0.870968, 0)
 to[out = -90, in = -90, distance = 1.2161] (0.967742, 0)
 to[out = 90, in = 90, distance = 6.0805] (0.483871, 0)
 to[out = -90, in = -90, distance = 0.405367] (0.516129, 0)
 to[out = 90, in = 90, distance = 5.26977] (0.935484, 0)
 to[out = -90, in = -90, distance = 0.405367] (0.903226, 0)
 to[out = 90, in = 90, distance = 1.2161] (0.806452, 0)
 to[out = -90, in = -90, distance = 0.405367] (0.774194, 0)
 to[out = 90, in = 90, distance = 1.2161] (0.677419, 0)
 to[out = -90, in = -90, distance = 6.0805] (0.193548, 0)
 to[out = 90, in = 90, distance = 0.405367] (0.225806, 0)
 to[out = -90, in = -90, distance = 5.26977] (0.645161, 0)
 to[out = 90, in = 90, distance = 1.2161] (0.548387, 0)
 to[out = -90, in = -90, distance = 2.83757] (0.322581, 0)
 to[out = 90, in = 90, distance = 2.83757] (0.0967742, 0)
 to[out = -90, in = -90, distance = 0.405367] (0.0645161, 0)
 to[out = 90, in = 90, distance = 3.6483] (0.354839, 0)
 to[out = -90, in = -90, distance = 1.2161] (0.451613, 0)
to[out = 90, in = 180, distance = 6.89124] (0.965388, 0.182796);
\draw[help lines] (0.5, 0) circle (0.5);
\draw[fill] (0.0346123, 0.182796) circle (0.0129099);
\draw[fill] (0.965388, 0.182796) circle (0.0129099);
\draw[fill] (0, 0) circle (0.0129099);
\draw[fill] (1, 0) circle (0.0129099);
\end{tikzpicture}
\hspace{0.7cm}
\begin{tikzpicture}[scale = 4.3]
    \node (a) at (0.5, -0.6) {(d)};
\draw[thick] (0, 0) to (1, 0);
\draw[ultra thick] (0.0359937, 0.186275) to[out = 0, in = 90, distance = 5.54399] (0.441176, 0)
 to[out = -90, in = -90, distance = 0.369599] (0.411765, 0)
 to[out = 90, in = 90, distance = 4.80479] (0.0294118, 0)
 to[out = -90, in = -90, distance = 1.1088] (0.117647, 0)
 to[out = 90, in = 90, distance = 2.58719] (0.323529, 0)
 to[out = -90, in = -90, distance = 3.32639] (0.588235, 0)
 to[out = 90, in = 90, distance = 0.369599] (0.617647, 0)
 to[out = -90, in = -90, distance = 4.06559] (0.294118, 0)
 to[out = 90, in = 90, distance = 1.1088] (0.205882, 0)
 to[out = -90, in = -90, distance = 6.28318] (0.705882, 0)
 to[out = 90, in = 90, distance = 0.369599] (0.735294, 0)
 to[out = -90, in = -90, distance = 1.1088] (0.823529, 0)
 to[out = 90, in = 90, distance = 0.369599] (0.852941, 0)
 to[out = -90, in = -90, distance = 1.1088] (0.941176, 0)
 to[out = 90, in = 90, distance = 5.54399] (0.5, 0)
 to[out = -90, in = -90, distance = 0.369599] (0.529412, 0)
 to[out = 90, in = 90, distance = 4.80479] (0.911765, 0)
 to[out = -90, in = -90, distance = 0.369599] (0.882353, 0)
 to[out = 90, in = 90, distance = 1.1088] (0.794118, 0)
 to[out = -90, in = -90, distance = 0.369599] (0.764706, 0)
 to[out = 90, in = 90, distance = 1.1088] (0.676471, 0)
 to[out = -90, in = -90, distance = 5.54399] (0.235294, 0)
 to[out = 90, in = 90, distance = 0.369599] (0.264706, 0)
 to[out = -90, in = -90, distance = 4.80479] (0.647059, 0)
 to[out = 90, in = 90, distance = 1.1088] (0.558824, 0)
 to[out = -90, in = -90, distance = 2.58719] (0.352941, 0)
 to[out = 90, in = 90, distance = 3.32639] (0.0882353, 0)
 to[out = -90, in = -90, distance = 0.369599] (0.0588235, 0)
 to[out = 90, in = 90, distance = 4.06559] (0.382353, 0)
 to[out = -90, in = -90, distance = 1.1088] (0.470588, 0)
 to[out = 90, in = 90, distance = 6.28318] (0.970588, 0)
 to[out = -90, in = -90, distance = 9.97918] (0.176471, 0)
 to[out = 90, in = 90, distance = 0.369599] (0.147059, 0)
to[out = -90, in = 180, distance = 9.2] (0.9, -0.3);
\draw[help lines] (0.5, 0) circle (0.5);
\draw[fill] (0.0359937, 0.186275) circle (0.0129099);
\draw[fill] (0.9, -0.3) circle (0.0129099);
\draw[fill] (0, 0) circle (0.0129099);
\draw[fill] (1, 0) circle (0.0129099);
\end{tikzpicture}
    \caption{Constructing an almost irreducible meander for odd $n$.}
    \label{fig: construct almost irreducible odd}
\end{figure}

    Let us formalize this procedure. Let $M$ be an arbitrary non-singular meander of total order $n$ with permutation $(a_1,\dots,a_{n})$. If $n$ is odd, consider a non-singular meander $M'$ of total order $2n+20$ with permutation (recall that non-singular meanders are uniquely determined by their permutation):
\begin{align*}
    (&13,\ 12,\ 1,\ 4,\ 9,\ 18,\ 19,\ 8,\ 5, \\
     &20+2a_1,\ 20+2a_2 - 1,\ 20+2a_3,\ 20+2a_4 - 1,\ \dots,\ 20+2a_n,\\
     &15,\ 16,\\
     &20+2a_n-1,\ 20+2a_{n-1},\ 20+2a_{n-2}-1,\ 20+2a_{n-3},\ \dots,\ 20+2a_1-1,\\
     &6,\ 7,\ 20,\ 17,\ 10,\ 3,\ 2,\ 11,\ 14).
\end{align*}
The only submeanders of $M'$ are meanders of total order two (this follows from Lemma~\ref{lem: submeanders and permutations}), and thus $M'$ is almost irreducible. 
We also can consider non-singular meander $M''$ of total order $2n + 23$ with permutation 
\begin{align*}
    (&15,\ 14,\ 1,\ 4,\ 11,\ 20,\ 21,\ 10,\ 7, \\
     &22+2a_1,\ 22+2a_2 - 1,\ 22+2a_3,\ 22+2a_4-1,\ \dots,\ 22+2a_n,\\
     &17,\ 18,\\
     &22+2a_n-1,\ 22+2a_{n-1},\ 22+2a_{n-2}-1,\ 22+2a_{n-3},\ \dots,\ 22+2a_1-1,\\
     &8,\ 9,\ 22,\ 19,\ 12,\ 3,\ 2,\ 13,\ 16,\\
     &2n + 23,\ 6,\ 5).
\end{align*}
The same argument shows that $M''$ is also almost irreducible.

If $n$ is even, two non-singular meanders $M'$ and $M''$ with permutations 
\begin{align*}
    (&15,\ 14,\ 1,\ 4,\ 11,\ 18,\ 19,\ 10,\ 7, \\
     &20+2a_1,\ 20+2a_2-1,\ 20+2a_3,\ 20+2a_4-1,\ \dots,\ 20+2a_n-1,\\
     &6,\ 5,\\
     &20+2a_n,\ 20+2a_{n-1}-1,\ 20+2a_{n-2},\ 20+2a_{n-3}-1,\ \dots,\ 20+2a_1-1,\\
     &8,\ 9,\ 20,\ 17,\ 12,\ 3,\ 2,\ 13,\ 16)
\end{align*}
and 
\begin{align*}
    (&17,\ 16,\ 1,\ 4,\ 13,\ 20,\ 21,\ 12,\ 9, \\
     &22+2a_1,\ 22+2a_2-1,\ 22+2a_3,\ 22+2a_4-1,\ \dots,\ 22+2a_n-1,\\
     &8,\ 7,\\
     &22+2a_n,\ 22+2a_{n-1}-1,\ 22+2a_{n-2},\ 22+2a_{n-3}-1,\ \dots,\ 22+2a_1-1,\\
     &10,\ 11,\ 22,\ 19,\ 14,\ 3,\ 2,\ 15,\ 18,\\
     &2n + 23,\ 6,\ 5)
\end{align*}
are almost irreducible of total order $2n+20$ and $2n+23$ respectively.

Thus we have the following inequalities
\begin{align*}
    \liminf_{n\to \infty} \sqrt[n]{\A_n} \geq \lim_{n\to \infty} \sqrt[n]{\M_{\frac{n-23}{2}, 0}} = \sqrt{\lim_{n\to \infty} \sqrt[n]{\M_{n,0}}}.
\end{align*}

The results of~\cite{AP05} imply that $\lim\limits_{n\to \infty} \sqrt[n]{{\M}_{n, 0}}\geq \sqrt{11.38}$, and we finally get
\begin{equation*}
     \liminf_{n\to \infty} \sqrt[n]{\A_n} \geq \sqrt[4]{11.38} \approx 1.83669.
\end{equation*}
\end{proof}

\bibliography{biblio}
\bibliographystyle{alpha}

\newpage
\appendix
\section{Table with meander numbers} \label{app: table}
\begin{table}[h!]
\begin{tabular}{c|ccc}
Order $n$ & $\M_{n,0}$ & $\MIS_{n,0}$ & $\A_n$ \\
\hline
1 & 1 & 0 & 0\\
2 & 1 & 1 & 0\\
3 & 2 & 2 & 0\\
4 & 3 & 3 & 0\\
5 & 8 & 8 & 0\\
6 & 14 & 14 & 0\\
7 & 42 & 42 & 0\\
8 & 81 & 79 & 2\\
9 & 262 & 252 & 2\\
10 & 538 & 494 & 0\\
11 & 1828 & 1636 & 0\\
12 & 3926 & 3294 & 26\\
13 & 13820 & 11188 & 36\\
14 & 30694 & 22952 & 52\\
15 & 110954 & 79386 & 64\\
16 & 252939 & 165127 & 516\\
17 & 933458 & 579020 & 816\\
18 & 2172830 & 1217270 & 2186\\
19 & 8152860 & 4314300 & 3296\\
20 & 19304190 & 9146746 & 15054\\
21 & 73424650 & 32697920 & 24946\\
22 & 176343390 & 69799476 & 84090\\
23 & 678390116 & 251284292 & 138352\\
24 & 1649008456 & 539464358 & 544652\\
25 & 6405031050 & 1953579240 & 934450\\
26 & 15730575554 & 4214095612 & 3377930\\
27 & 61606881612 & 15336931928 & 5831520\\
28 & 152663683494 & 33218794236 & 22075152\\
29 & 602188541928 & 121416356108 & 38959552\\
30 & 1503962954930 & 263908187100 & 143815358\\
31 & 5969806669034 & 968187827834 & 256128664\\
32 & 15012865733351 & 2110912146295 & 959463704\\
33 & 59923200729046 & 7769449728780 & 1732188588\\
34 & 151622652413194 & 16985386737830 & 6440145162\\
35 & 608188709574124 & 62696580696172 & 11727449592\\
36 & 1547365078534578 & 137394914285538 & 43825381338\\
37 & 6234277838531806 & 508451657412496 & 80571300722\\
38 & 15939972379349178 & 1116622717709012 & 300477174306\\
\end{tabular}
\caption{Meander numbers.}
\label{tab: numeric results}
\end{table}
\end{document}